\documentclass[11pt]{amsart}

\usepackage[letterpaper,margin=.6in]{geometry} 
\usepackage[colorlinks,
             linkcolor=blue,
             citecolor=black!75!red,
             pdfproducer={pdfLaTeX},
             pdfpagemode=None,
             bookmarksopen=true
             bookmarksnumbered=true]{hyperref}
\usepackage{parskip}
\usepackage[matrix,arrow,ps,color,line,curve,frame,all]{xy}

\usepackage{tikz, tikz-cd}
\usetikzlibrary{matrix,arrows,decorations.pathmorphing, fit}
\tikzset{myboxgroup/.style={draw, densely dotted}} 

\usepackage{amsmath,amsthm,stmaryrd}
\usepackage{cleveref}

\newtheorem{lemma}{Lemma}[section]
\newtheorem{theorem}[lemma]{Theorem}
\newtheorem*{theorem*}{Theorem}
\newtheorem{proposition}[lemma]{Proposition}
\newtheorem{corollary}[lemma]{Corollary}

\theoremstyle{definition} 

\newtheorem{definitionnodiamond}[lemma]{Definition}
\newtheorem{examplenodiamond}[lemma]{Example}
\newtheorem{remarknodiamond}[lemma]{Remark}

\newenvironment{remark}{\begin{remarknodiamond}}{\hfill\ensuremath\blacklozenge\end{remarknodiamond}}

\makeatletter
\let\xx@thm\@thm
\AtBeginDocument{\let\@thm\xx@thm}
\makeatother

\numberwithin{equation}{section}

\newcounter{stepofproof}

\crefname{section}{Section}{Sections}
\crefformat{section}{#2Section~#1#3} 
\Crefformat{section}{#2Section~#1#3} 

\crefname{subsection}{}{Subsections}
\crefformat{subsection}{\S#2#1#3} 
\Crefformat{subsection}{\S#2#1#3}

\crefname{definition}{Definition}{Definitions}
\crefformat{definition}{#2Definition~#1#3} 
\Crefformat{definition}{#2Definition~#1#3} 

\crefname{example}{Example}{Examples}
\crefformat{example}{#2Example~#1#3} 
\Crefformat{example}{#2Example~#1#3} 

\crefname{examplenodiamond}{Example}{Examples}
\crefformat{examplenodiamond}{#2Example~#1#3} 
\Crefformat{examplenodiamond}{#2Example~#1#3} 

\crefname{remark}{Remark}{Remarks}
\crefformat{remark}{#2Remark~#1#3} 
\Crefformat{remark}{#2Remark~#1#3} 

\crefname{remarknodiamond}{Remark}{Remarks}
\crefformat{remarknodiamond}{#2Remark~#1#3} 
\Crefformat{remarknodiamond}{#2Remark~#1#3} 

\crefname{convention}{Convention}{Conventions}
\crefformat{convention}{#2Convention~#1#3} 
\Crefformat{convention}{#2Convention~#1#3} 

\crefname{lemma}{Lemma}{Lemmas}
\crefformat{lemma}{#2Lemma~#1#3} 
\Crefformat{lemma}{#2Lemma~#1#3} 

\crefname{proposition}{Proposition}{Propositions}
\crefformat{proposition}{#2Proposition~#1#3} 
\Crefformat{proposition}{#2Proposition~#1#3} 

\crefname{corollary}{Corollary}{Corollaries}
\crefformat{corollary}{#2Corollary~#1#3} 
\Crefformat{corollary}{#2Corollary~#1#3} 

\crefname{theorem}{Theorem}{Theorems}
\crefformat{theorem}{#2Theorem~#1#3} 
\Crefformat{theorem}{#2Theorem~#1#3} 

\crefname{assumption}{Assumption}{Assumptions}
\crefformat{assumption}{#2Assumption~#1#3} 
\Crefformat{assumption}{#2Assumption~#1#3} 

\crefname{equation}{}{}
\crefformat{equation}{(#2#1#3)} 
\Crefformat{equation}{(#2#1#3)}

\crefname{align}{}{}
\crefformat{align}{(#2#1#3)} 
\Crefformat{align}{(#2#1#3)}

\crefname{proofstep}{Step}{Steps}
\crefformat{proofstep}{#2Step~#1#3} 
\Crefformat{proofstep}{#2Step~#1#3}

\newcommand\cat[1]{\textsc{#1}}

\newcommand\arXiv[1]{\href{http://arxiv.org/abs/#1}{\nolinkurl{arXiv:#1}}}
\newcommand\MRnumber[1]{\href{http://www.ams.org/mathscinet-getitem?mr=#1}{\nolinkurl{MR#1}}}
\newcommand\DOI[1]{\href{http://dx.doi.org/#1}{\nolinkurl{DOI:#1}}}
\newcommand\MAILTO[1]{\href{mailto:#1}{\nolinkurl{#1}}}

\usepackage{amsfonts,amssymb}

\newcommand\fb{\mathfrak b}

\newcommand\fsl{{\mathfrak s}{\mathfrak l}}

\def\sT{{\sf T}}

\newcommand\bC{\mathbb C}

\newcommand\bG{\mathbb G}

\newcommand\bP{\mathbb P}

\def\a{\alpha}
\def\b{\beta}
\def\c{\gamma}
\def\d{\delta}

\def\l{\lambda}

\def\s{\sigma}

\def\ve{\varepsilon}

\def\CC{{\mathbb C}}

\def\GG{{\mathbb G}}

\def\PP{{\mathbb P}}

\def\NN{{\mathbb N}}
\def\PP{{\mathbb P}}

\def\ZZ{{\mathbb Z}}

\newcommand\cL{\mathcal L}
\newcommand\cM{\mathcal M}

\newcommand\cO{\mathcal O}
\newcommand\cP{\mathcal P}

\DeclareMathOperator\qcoh{\cat{Qcoh}}

\newcommand\dg{\text{deg}}

\def\GL{{\sf GL}}
\def\Gr{{\sf Gr}}
\def\QGr{\operatorname{\sf QGr}}
\def\Proj{\operatorname{Proj}}
\def\Projnc{\operatorname{Proj}_{nc}}
\def\Fdim{{\sf Fdim}}
\def\Mod{{\sf Mod}}
\def\pd{{\operatorname {\partial}}}
\def\coker{\operatorname {coker}}

\renewcommand\lim{\varprojlim}

\numberwithin{equation}{section}

\newenvironment{pf}{\noindent{\bf Proof.}}{\hfill $\square$\medskip}

\newcommand{\inner}[1]{\left\langle #1 \right\rangle}

\newcommand{\Uq}{U_q(\fsl(2,\CC))}

\makeatletter
\def\l@section{\@tocline{1}{0pt}{1pc}{}{}}
\def\l@subsection{\@tocline{2}{0pt}{1pc}{4.6em}{}}
\def\l@subsubsection{\@tocline{3}{0pt}{1pc}{7.6em}{}}
\renewcommand{\tocsection}[3]{%
  \indentlabel{\@ifnotempty{#2}{\makebox[2.3em][l]{%
    \ignorespaces#1 #2.\hfill}}}#3}
\renewcommand{\tocsubsection}[3]{%
  \indentlabel{\@ifnotempty{#2}{\hspace*{2.3em}\makebox[2.3em][l]{%
    \ignorespaces#1 #2.\hfill}}}#3}
\renewcommand{\tocsubsubsection}[3]{%
  \indentlabel{\@ifnotempty{#2}{\hspace*{4.6em}\makebox[3em][l]{%
    \ignorespaces#1 #2.\hfill}}}#3}
\makeatother

\title[Homogenization of the quantized  enveloping algebra $\protect\Uq$]
{Non-commutative Geometry of Homogenized Quantum $\fsl(2,\CC)$}
 
\author[Alex Chirvasitu]{Alex Chirvasitu}
\author[S. Paul Smith]{S. Paul Smith}
\author[Liang Ze Wong]{Liang Ze Wong} 	 

\address{Department of Mathematics, Box 354350, University of  Washington, Seattle, WA 98195,USA.}

\email{chirva@math.washington.edu, smith@math.washington.edu, wonglz@uw.edu} 



\begin{document}

\begin{abstract}
This paper examines the relationship between certain non-commutative analogues of projective 3-space, $\mathbb{P}^3$, and the quantized enveloping algebras $U_q(\mathfrak{sl}_2)$.
The relationship is mediated by certain non-commutative graded algebras $S$, one for each $q \in \mathbb{C}^\times$, having a degree-two central element $c$ such
that $S[c^{-1}]_0 \cong U_q(\mathfrak{sl}_2)$. The non-commutative analogues of $\mathbb{P}^3$ are the spaces $\operatorname{Proj}_{nc}(S)$. We show how the points, fat points, lines, and quadrics, in $\operatorname{Proj}_{nc}(S)$, and their incidence relations,  correspond to finite dimensional irreducible representations of $U_q(\mathfrak{sl}_2)$, Verma modules, annihilators of Verma modules, and homomorphisms between them.
\end{abstract}
                             
\maketitle

\setcounter{tocdepth}{2} 
\tableofcontents

\pagenumbering{arabic}

\section{Introduction}

This paper concerns the interplay between the geometry of some non-commutative analogues of $\PP^3$ and the representation theory of  the quantized enveloping algebras, $U_q(\fsl_2)$, of $\fsl(2,\CC)$. We always assume that $q$ is not a root of unity.

\subsection{$\Projnc(S)$ and $U_q(\fsl_2)$}
In \S\ref{ssect.S}, we define a family of non-commutative graded $\CC$-algebras $S=\CC [E,F,K,K']$ depending on a parameter $q \in \CC-\{0,\pm 1,\pm i\}$ 
that have the same Hilbert series and the ``same'' homological properties as the polynomial ring in 4 variables. 
For these reasons the non-commutative spaces $\Projnc(S)$ have much in common with $\PP^3$.
The element $KK'$ belongs to the center of $S$ and $S[(KK')^{-1}]_0  \cong U_q(\fsl_2)$. 
Thus, $U_q(\fsl_2)$ is the coordinate ring of the ``open complement'' to the union of the ``hyperplanes'' $\{K=0\}$ and $\{K'=0\}$ in $\Projnc(S)$.
 This analogy can be formalized: there is an abelian category $\qcoh(\cdot)$, defined below, that plays the role of the category of quasi-coherent 
 sheaves and an adjoint pair of functors 
\begin{equation}
\label{j*j*}
\xymatrix{
\qcoh(\Projnc(S)) \ar@<1ex>[r]^-{j^*}  & \ar@<1ex>[l]^-{j_*}  \Mod(U_q(\fsl_2))
}
\end{equation}
that behave like the inverse and direct image functors for an open immersion  $j:\PP^3- \{\hbox{two planes}\} \to \PP^3$.

\subsubsection{}
By definition, $\qcoh(\Projnc(S))$ is  the quotient category
$$
\QGr(S) \; := \; \frac{\Gr(S)}{\Fdim(S)}
$$
where $\Gr(S)$ denotes the category of $\ZZ$-graded left $S$-modules and $\Fdim(S)$ denotes the full subcategory of $\Gr(S)$ consisting of those modules that are the sum of their finite dimensional submodules. If $S$ were the polynomial ring on 4 variables, then the category $\QGr(S)$ would be 
equivalent to $\qcoh(\PP^3)$, the category of quasi-coherent sheaves on $\PP^3$, and this equivalence would send a graded module $M$ to the $\cO_{\PP^3}$-module that Hartshorne denotes by $\widetilde{M}$.

\subsubsection{}
\label{ssect.localize}
There is an exact functor $\pi^*:\Gr(S) \to \QGr(S)$ that sends a graded $S$-module $M$ to $M$ viewed as an object in $\QGr(S)$.
 The composition
 \begin{equation}
\label{j*pi*}
\xymatrix{
\Gr(S)  \ar[r]^-{\pi^*}    & \QGr(S)=  \qcoh(\Projnc(S)) \ar[r]^-{j^*}  &  \Mod(U_q(\fsl_2))
}
\end{equation}
sends a graded $S$-module $M$ to $M[(KK')^{-1}]_0$.

\subsubsection{}
A guiding theme of this paper is the interaction between non-commutative geometry (where $\QGr(S)$ belongs) and representation theory (where $\Mod(U_q(\fsl_2))$ belongs). 
We show how the points, fat points, lines, and quadrics, in $\Projnc(S)$, and their incidence relations,  
correspond to finite dimensional irreducible representations of $U_q(\fsl_2)$, Verma modules, annihilators of Verma modules, and homomorphisms between them.  

Just as passing from affine to projective geometry provides a more elegant picture that unifies seemingly different objects (affine vs. projective conic sections, for example), passing from the ``affine'' category $\Mod(U_q(\fsl_2))$ to the ``projective'' category 
$\qcoh(\Projnc(S))$ results in a more complete picture of $\Mod(U_q(\fsl_2))$.

\subsection{Lines and Verma modules, fat points and finite-dimensional irreps}
The most important $U_q(\fsl_2)$-modules are its finite dimensional irreducible representations and its Verma modules. 
In \S\ref{sec.Uq}, we show that if $V$ is a Verma module for $U_q(\fsl_2)$, there is a graded $S$-module $M$ such that:
\begin{enumerate}
  \item 
  $V \cong j^*\pi^* M$;
  \item 
  $M \cong S/S\ell^\perp$ where $\ell^\perp \subseteq S_1$ is a codimension-two subspace;
  \item
  $\dim(M_i)=i+1$ for all $i \ge 0$, i.e., $M$ has the same Hilbert series as the polynomial ring on two variables;
  \item
  $M$ is a line module for $S$;
  \item
  $M$ is a Cohen-Macaulay $S$-module.
\end{enumerate}  
In \S\ref{sec.Uq}, we also show that if  $L$ is a finite dimensional irreducible representation of $U_q(\fsl_2)$, there is a graded $S$-module $F$ having the following properties:
\begin{enumerate}
  \item 
  $L \cong j^*\pi^* F$;
  \item 
   $\dim(F_i)=\dim(L)$ for all $i \ge 0$ and    $\dim(F_i)=0$ for all $i < 0$;
  \item 
  every proper quotient of $F$ is finite-dimensional, which implies that $F$ is a simple object in $\qcoh(\Projnc(S))$;
  \item
  $F$ is a fat point module for $S$;
   \item
  $F$ is a Cohen-Macaulay $S$-module.
\end{enumerate}  
Items (4) are, essentially, definitions; see \S\ref{ssect.linear.mods}.

\subsubsection{Point modules and line modules}
Let $R$ be the polynomial ring on 4 variables with its standard grading.
The points in $\PP^3= \Proj(R)$ are in bijection with the modules $R/I$ such that $\dim(R_i/I_i)=1$ for all $i\ge 0$.
 The lines in $\PP^3 $ are in bijection with the modules $R/I$ such that $\dim(R_i/I_i)=i+1$ for all $i\ge 0$. 
 
 If $S$ is one of the algebras in  \S\ref{ssect.S}, a graded $S$-module $M$ is called a {\it point module}, resp. a {\it line module}, if it is isomorphic to 
$S/I$ for some left ideal $I$ such that $\dim(S_i/I_i)=1$, resp.  $\dim(S_i/I_i)=i+1$, for all $i \ge 0$. 

 There are fine moduli spaces that parametrize the point modules and line modules for $S$. These fine moduli spaces are called the 
 {\it point scheme} and {\it line scheme}
 respectively. The point scheme for $S$ is a closed subscheme of $\PP^3=\PP(S_1^*)$ and the line scheme for $S$ is a closed subscheme of the Grassmannian
 $\GG(1,3)$ consisting of the lines in $\PP^3$.
 
In \S\ref{sec.S}, we determine the line modules and the point modules for $S$.

\subsubsection{The point modules for $S$}
\label{sect.pt.mods.quadrics}
The point scheme, $\cP_S$, for $S$ is  $C\cup C' \cup L \cup \{p_1,p_2\} \subseteq \PP(S_1^*) = \PP^3$, the union of two plane conics, $C$ and $C'$, meeting at two points, the line $L$ through those two points, and two additional points (\Cref{thm.pts.S}).  
If $M_p=S/Sp^\perp$ is the point module corresponding to $p \in \cP_S$, then $(M_p)_{\ge 1}$ is a shifted point module; i.e., 
$(M_p)_{\ge 1}(1)$ is a point module and therefore isomorphic to $M_{p'}$ for some point $p' \in \cP_S$. 
General results show there is an automorphism $\s:\cP_S \to \cP_S$ such that $p'=\sigma^{-1}p$. Thus, $(M_p)_{\ge 1} \cong M_{\s^{-1}p}(-1)$.
We determine $\cP_S$ and $\s$ in \S\ref{sec.S}.

\subsubsection{The line modules for $S$}
\Cref{thm.lines.D} says that the lines $\ell \subseteq \PP^3=\PP(S_1^*)$ for which $S/S\ell^\perp$ is a line module are precisely those lines that meet $C \cup C'$
with multiplicity two; i.e., the secant lines to $C \cup C'$. These are exactly the lines lying on a certain pencil of quadrics $Q(\l) \subseteq \PP^3$, $\l \in \PP^1$.
This should remind the reader of the analogous result for the 4-dimensional Sklyanin algebras in which the lines in $\PP^3$ that correspond to line modules
are exactly the secant lines to the quartic elliptic curve $E$.

The labelling of the line modules is such that the Verma module $M(\l)$ is isomorphic to $j^*\pi^*(S/S\ell^\perp)$ for a unique line $\ell \subseteq Q(\l)$.

\subsubsection{Incidence relations}
If $(p)+(p')$ is a degree-two divisor on $C \cup C'$, we write $M_{p,p'}$ for the line module $S/S\ell^\perp$ where $\ell^\perp$ is the subspace of $S_1$
that vanishes on the line $\ell \subseteq \PP^3=\PP(S_1^*)$ whose scheme-theoretic intersection with $C \cup C'$ is $(p)+(p')$. \Cref{lem.incidence} 
shows there is an exact sequence 
$$
0 \to M_{\sigma p,\sigma^{-1}p'}(-1) \to M_{p,p'} \to M_p \to 0.
$$
\Cref{prop.pt.line.incidence} shows that if the line $\ell$ just referred to meets the line $\{K=K'=0\} \subseteq \cP_S$ at a point $p''$, there is an 
exact sequence 
$$
0 \to M_{\sigma^{-1} p,\sigma^{-1}p'}(-1) \to M_{p,p'} \to M_{p''} \to 0.
$$

\subsubsection{Finite dimensional simple modules}
Let $n\in \NN$. If $q$ is not a root of unity there are two simple $U_q(\fsl_2)$-modules of dimension $n+1$. We label them $L(n,\pm)$ in such a way that
there are exact sequences 
\begin{equation}
\label{verma.ses}
0 \to M(\pm q^{-n-2}) \to M(\pm q) \to L(n,\pm) \to 0
\end{equation}
in which $M(\l)$ denotes the Verma module of highest weight $\l$. 

In \S\ref{sec.Uq} we show there are $S$-modules $V(n,\pm)$ that are also $S[(KK')^{-1}]$-modules, and hence modules over $S[(KK')^{-1}]_0  \cong U_q(\fsl_2)$
and, as such, $V(n,\pm) \cong L(n,\pm)$. We define graded $S$-modules $F(n,\pm)$ such that $F(n,\pm)[(KK')^{-1}]_0 \cong L(n,\pm)$; i.e., if we view $F(n,\pm)$ as
an object in $\QGr(S)$, then
$$
j^* F(n,\pm) \cong L(n,\pm).
$$
Furthermore, we show there are exact sequences 
\begin{equation}
\label{line.fat.pt.ses}
0 \to M_{\ell'_{\pm}}(-n-1) \to M_{\ell_{\pm}} \to F(n,\pm) \to 0
\end{equation}
in $\qcoh(\Projnc(S))$ and that 
(\ref{verma.ses}) is obtained from (\ref{line.fat.pt.ses}) by applying the functor $j^*$, i.e., by restricting the exact sequence (\ref{line.fat.pt.ses}) 
in $\qcoh(\Projnc(S))$ to the ``open affine subscheme'' $\{KK' \ne 0\}$. 
Here $M_{\ell_{\pm}}$ denotes the line module $S/S\ell_\pm^\perp$ corresponding to a line $\ell_\pm\subseteq \PP(S_1^*)=\PP^3$.

\subsubsection{Heretical Verma modules} 
The connections we establish between Verma modules and line modules highlights one way in which the $q$-deformation $U_q(\fsl_2)$ is ``more rigid'' or ``less symmetric'' than the enveloping algebra $U(\fsl_2)$: there is a $\PP^1$-family of Borel subalgebras of
 $\fsl_2$, but there are only two reasonable candidates for the role of the quantized enveloping algebra of a ``Borel subalgebra'' of quantum $\fsl_2$. 

Fix one of the two ``Borel subalgebras'',  $U_q(\mathfrak{b}) \subseteq U_q(\fsl_2)$. It gives rise by induction to Verma modules 
$M_{\mathfrak{b}}(\lambda)=U_q(\fsl_2) \otimes_{U_q(\fb)} \CC_\l$, $\l \in \CC^\times$.  Thus, one obtains two 1-parameter families of Verma modules for $U_q(\fsl_2)$. In sharp contrast, by varying both the Borel subalgebra and the highest weight one obtains a 2-parameter family of Verma modules for $U(\fsl_2)$. 
Our perspective on $U_q(\fsl_2)$ as a non-commutative open subscheme of a non-commutative $\PP^3$ allows us to fit the two 1-parameter families of Verma modules for $U_q(\fsl_2)$ into a single 2-parameter family of modules, thus undoing the rigidification phenomenon alluded to in the previous paragraph.
 It is these additional Verma-like modules that we refer to as `heretical' in the title of this subsection.  
 
 For simplicity of discussion, fix a finite dimensional simple module $L(n,+)$ and the corresponding fat point module $F(n,+)$ for which $j^*F(n,+) \cong L(n,+)$. 
 The module $L(n,+)$ appears in exactly two sequences of the form 
 (\ref{verma.ses}), one for each ``Borel subalgebra'' of  $U_q(\fsl_2)$; in contrast, $F(n,+)$ appears in a 1-parameter family of  sequences of the form 
 (\ref{line.fat.pt.ses}), one for each line in one of the rulings on the quadric $Q(q^n)$. 
 Likewise, a fixed finite dimensional simple $U(\fsl_2)$-module fits into a 1-parameter family of  sequences of the form (\ref{verma.ses}). 
 If we broadened the definition of a Verma module for $U_q(\fsl_2)$ so as to include $j^*M_\ell$ for all line modules $M_\ell$ one would then obtain a 
 1-parameter family of  sequences of the form (\ref{verma.ses}).

\subsubsection{Annihilators of Verma modules and quadrics in $\Projnc(S)$}
When $q$ is not a root of unity, the center of $U_q(\fsl_2)$ is generated by a single central element $C$ called the Casimir element. 
A Verma module is  annihilated by $C-\nu$ for a unique $\nu \in \CC$ and given $\nu$ there are, usually, four Verma modules annihilated by 
$C-\nu$. 

There is a non-zero central element $\Omega \in  S_2$ such that $C=\Omega(KK')^{-1}$ under the isomorphism $U_q(\fsl_2) \cong S[(KK')^{-1}]_0$.
A line module for $S$ is annihilated by $\Omega-\nu KK'$ for a  unique $\nu \in \CC\cup\{\infty\}$ and given $\nu$ there are, usually,
two 1-parameter families of line modules annihilated by $\Omega-\nu KK'$. There is an isomorphism
$$
\frac{S}{(\Omega - \nu KK')}[(KK')^{-1}]_0 \; \cong \; \frac{U_q(\fsl_2)}{(C-\nu)}
$$
and an adjoint pair of functors
\begin{equation}
\label{nc.quadric}
\qcoh\left(\Projnc\left(\frac{S}{(\Omega - \nu KK')}\right)\right) \xymatrix{ \ar@<1ex>[r]^-{j^*}  & \ar@<1ex>[l]^-{j_*} } \Mod\left(\frac{U_q(\fsl_2)}{(C-\nu)}\right)
\end{equation}

We think of $\Projnc({S}/{(\Omega - \nu KK')})$ as a non-commutative quadric hypersurface in $\Projnc(S)$ and think of $U_q(\fsl_2)/(C-\nu)$ as the coordinate ring of
a non-commutative affine quadric. Non-commutative quadrics in non-commutative analogues
of $\PP^3$ were examined in \cite{SVdB-NCQ}. The results there apply to the present situation.  The line modules for $S$ that are annihilated by $\Omega - \nu KK'$
provide rulings on the non-commutative quadric and the non-commutative quadric is smooth if and only if it has two rulings. 
We note that  $\Projnc({S}/{(\Omega - \nu KK')})$ is smooth if and only if $U_q(\fsl_2)/(C-\nu)$ has finite global dimension.

In \S\ref{sect.pt.mods.quadrics}, we mentioned the pencil of quadrics $Q(\l) \subseteq \PP^3$, $\l \in \PP^1$, that contain $C \cup C'$. The $Q(\l)$'s are commutative 
quadrics and should not be confused with the non-commutative ones in the previous paragraph. 
If $\ell$ is a line on $Q(\l)$, then $M_\ell = S/S\ell^\perp$ is a line module so is annihilated by $\Omega - \nu KK'$ for some $\nu \in \CC\cup \infty$.

\subsubsection{What happens for $U(\fsl_2)$?}
In \cite{LBS93}, LeBruyn and Smith consider a graded algebra $H(\fsl_2)$ that has a central element 
$t$ in $H_1$ such that $H[t^{-1}]_0$ is isomorphic to the enveloping algebra $U(\fsl_2)$. They call $H(\fsl_2)$ a {\it homogenization} of $U(\fsl_2)$,

Since the Hilbert series of $H$ equals that of the  polynomial ring in 4 variables with its standard grading, and since $H$ has ``all'' the good homological properties the polynomial ring does, they view $H$ as a homogeneous coordinate ring of a non-commutative analogue of $\PP^3$, denoted by $\Projnc(H)$.  
Because $H[t^{-1}]_0 \cong U(\fsl_2)$, there is an adjoint pair of functors $j^*$ and $j_*$ fitting into diagrams like those in (\ref{j*j*}) and  (\ref{j*pi*}). Because $t$ has degree-one,  $j^*$ and $j_*$ behave like the inverse and direct image functors associated to the open complement to the hyperplane at infinity in $\PP^3$.
LeBruyn and Smith examine the point and line modules for $H$ and show that these modules are related to the finite dimensional irreducible representations and Verma modules for $\fsl_2$.  The situation for $U(\fsl_2)$ is simpler than that for $U_q(\fsl_2)$.

\subsubsection{Richard Chandler's results}
We are not the first to compute the point modules and line modules for $S$. 
Richard Chandler did this in his Ph.D. thesis \cite{Chandler}. His approach differs from ours. 
Following a method introduced by Shelton and Vancliff in \cite{ShV02_bis}, he uses {\it Mathematica} to compute 
a system of 45 quadratic polynomials in the Pl\"ucker coordinates on the Grassmanian $\GG(1,3)$,
the common zero locus of which is the line scheme for $S$. In contrast, we use the results on central extensions in \cite{LBSvdB}
to determine which lines in $\PP^3$ correspond to line modules. The two approaches are complementary.

\subsection{The structure of the paper}
In \S\ref{sec.Prelim}, we define the algebra $S$, the central focus of our paper, and discuss its position as a degenerate version of the 4-dimensional Sklyanin algebra and a homogenization of $U_q(\fsl_2)$. 
We introduce the category $\QGr(S)$ and its non-commutative geometry. We focus on point, line, and fat point, modules.   

In \S\ref{sec.D}, we examine a Zhang twist  $D$ of $S$. It has the property that $\Gr(D)\equiv  \Gr(S)$. In fact, $D$ has a central element $z \in D_1$ such that $A = D/(z)$ is a 3-dimensional Artin-Schelter regular algebra, making $D$ a central extension of $A$. This allows us to use the results in \cite{LBSvdB} to determine the point and line modules of $D$ in terms of those for $A$.

In \S\ref{sec.S}, we translate the results about $D$ back to $S$.

In \S\ref{sec.Uq}, we relate our results about point and line modules for $S$ to results about the finite dimensional irreducible representations and Verma modules of $U_q(\fsl_2)$. The following table  summarizes some of these relations:
\begin{table}[htp]
\begin{center}
\begin{tabular}{|c|c|c|c|c|c|c|c|}
\hline
 $\Gr(S)$ & $\Projnc(S)$ &   $\Mod(U_q(\fsl_2))$  
\\
\hline
Point modules & Points &   Finite-dimensional irreducible modules
\\
\hline
Line modules & Lines &   Verma modules
\\
\hline
\end{tabular}
\end{center} 
\caption{Relation to $U_q(\fsl_2)$-modules}
\label{dictionary}
\end{table}

In \S\ref{sec.Sklyanin}, we show that some of our results can be obtained as ``degenerations'' of results in \cite{SS92,CS15-2,SSJ93} about the 4-dimensional Sklyanin algebra.

Fig \ref{fig.landscape} presents a summary of the algebras in this paper and their relationships to $S$.

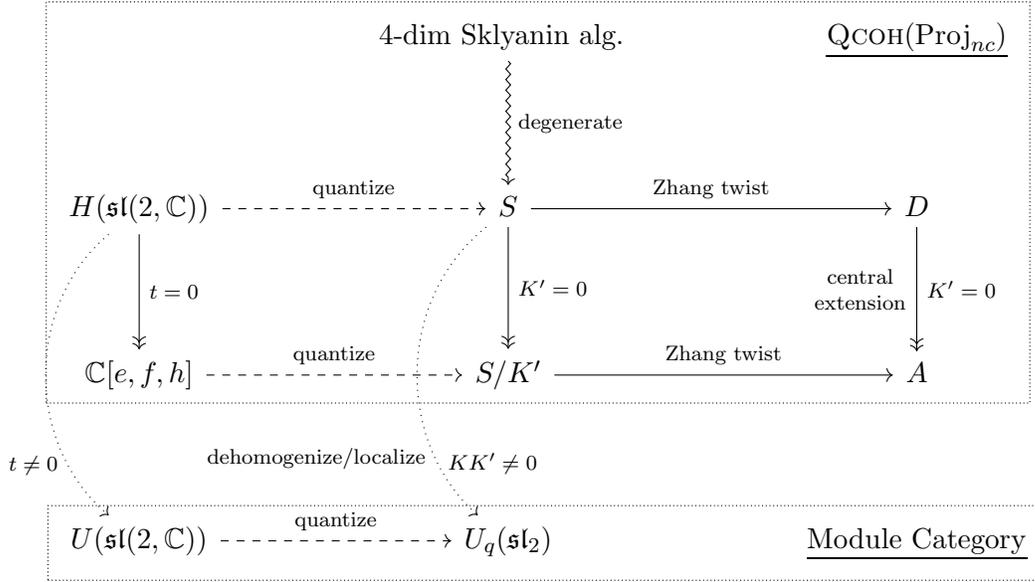
\begin{figure}
\hspace*{0.05\linewidth}
\begin{tikzpicture}[commutative diagrams/every diagram]
\matrix[matrix of math nodes, name=m, row sep=4em, column sep=5em, commutative diagrams/every cell] {
          & \text{4-dim Sklyanin alg. 
} & \underline{\textsc{Qcoh}(\operatorname{Proj}_{nc})}  \\
H(\mathfrak{sl}(2,\mathbb{C})) 
& S             & D \\
\mathbb{C}[e,f,h]          & S/K'            & A \\
U(\mathfrak{sl}(2,\mathbb{C}))    & U_q(\mathfrak{sl}_2)  & \underline{\text{Module Category}}              & \\
};
\node[myboxgroup, fit=(m-1-3) (m-2-1) (m-3-3)] {};
\node[myboxgroup, fit=(m-4-1) (m-4-3)] {};
\path[->, decoration={zigzag,segment length=4,amplitude=.9,
  post=lineto,post length=2pt}, font = \scriptsize]
  (m-1-2) edge[decorate] node[auto] {degenerate} (m-2-2);
\path[->,dashed, font=\scriptsize]
  (m-2-1) edge node[auto] {quantize}(m-2-2)
  (m-3-1) edge node[auto] {quantize}(m-3-2)
  (m-4-1) edge node[auto] {quantize}(m-4-2);  
\path[->, font=\scriptsize]
  (m-2-2) edge node[auto] {Zhang twist 
}(m-2-3)
  (m-3-2) edge node[auto] {Zhang twist} (m-3-3);
\path[->>, font = \scriptsize]
  (m-2-1) edge node[right]{$t = 0$} (m-3-1)
  (m-2-2) edge node[right]{$K'=0$} (m-3-2)
  (m-2-3) edge node[auto, left, align = center] {central \\ extension 
} node[right] {$K'=0$}(m-3-3);
\path[->, dotted, font = \scriptsize]  
  (m-2-1) edge [bend right=50] node[auto, pos = 0.85] {\hspace{15 mm}   dehomogenize/localize} node[left, pos = 0.8]{$t \neq 0$} (m-4-1)
  (m-2-2) edge [bend right=50] node[right, pos = 0.8]{$KK' \neq 0$} (m-4-2);
\end{tikzpicture}
    \caption{Algebras in this paper, their relationship to $S$, and their associated categories}
    \label{fig.landscape} 
\end{figure}

 \medskip
 \noindent
 1.4  {\bf Acknowledgements.}
 We thank Richard Chandler for sharing his results with us. We also
 thank him and Michaela Vancliff for useful conversations about his
 work and ours.

 We are also grateful to the anonymous referee's careful reading and
 comments, contributing to the improvement of the draft.

 A.C. was partially supported by NSF grant DMS-1565226.

\section{Preliminary notions}
\label{sec.Prelim}
\subsection{The category $\QGr$}
\label{ssec.QGr}
Let $\Bbbk$ be a field and $R$ a $\ZZ$-graded $\Bbbk$-algebra. The category $\QGr(R)$ is defined to be the quotient category
$$
\QGr(R) \; := \; \frac{\Gr(R)}{\Fdim(R)}\, ,
$$
where $\Gr(R)$ denotes the category of $\ZZ$-graded left $R$-modules with degree-preserving homomorphisms 
and $\Fdim(R)$ denotes the full subcategory of $\Gr(R)$ consisting of those modules that 
are the sum of their finite dimensional submodules.

The categories $\QGr(R)$ and $\Gr(R)$ have the same objects but different morphisms. 
There is an exact functor $\pi^*:\Gr(R) \to \QGr(R)$ that is the identity on objects. In the situations considered in this paper $\pi^*$ has a right adjoint $\pi_*$.
A morphism $f:M \to M'$ becomes an isomorphism in $\QGr(R)$, i.e., $\pi^*f$ is an isomorphism, 
if and only if $\ker(f)$ and $\coker(f)$ are in $\Fdim(R)$. In particular, a graded $R$-module is isomorphic to 0 in $\QGr(R)$ if and only if it is the sum of its finite dimensional modules. Two modules in $\Gr(R)$ are {\sf equivalent} if they are isomorphic in $\QGr(R)$.

If $M \in \Gr(R)$ and $n \in \ZZ$ we write $M(n)$ for the graded $R$-module that is $M$ as a left $R$-module but with new homogeneous components,
$M(n)_i=M_{n+i}$. 
This rule $M \rightsquigarrow M(n)$ extends to an auto-equivalence of $\Gr(R)$. Because it sends finite dimensional modules to finite dimensional modules,
it induces an auto-equivalence of $\QGr(R)$ that we denote by $\cM \rightsquigarrow \cM(n)$. 

If $M\in \Gr(R)$ we define $M_{\ge n} := M_n+M_{n+1}+ \cdots$. If $R=R_{\ge 0}$, then $M_{\ge n}$ is a submodule of $M$.

\subsection{Linear modules}
\label{ssect.linear.mods}
The importance of linear modules for non-commutative analogues of $\PP^n$ was first recognized by Artin, Tate, and Van den Bergh.
We recall a few notions from their papers \cite{ATV1,ATV2}.
 Let $M \in \Gr(R)$. If $M_n = 0$ for $n \ll 0$ and $\dim M_n < \infty$ for all $n$, the {\sf Hilbert series} of $M$ is the formal Laurent series 
\begin{equation*}
H_M(t) = \sum_n (\dim M_n) t^n.
\end{equation*} 
We are particularly interested in cyclic modules $M$ with Hilbert series having the form
\begin{equation*}
H_M(t) = \frac{n}{(1-t)^d}
\end{equation*}
for some $n,d \in \NN$. 
The  {\sf Gelfand-Kirillov (GK) dimension} of such a module is $d(M) = d$ and its {\sf multiplicity}  is $n$. 
If $d(M) = d$ and $d(M/N) < d$ for all non-zero submodules $N$, then $M$ is called {\sf $d$-critical}. Equivalent modules (in the sense of \Cref{ssec.QGr}) have the same GK-dimension, and also have the same multiplicity if they are not equivalent to $0$, so the notions of GK-dimension and multiplicity carry over to $\QGr(R)$ as well.

We call $M$ a {\sf linear} module if it is cyclic and its Hilbert series is $(1-t)^{-d}$. The cases  $d=1$ and $d=2$ play a key role: we call a linear module $M$ a 
\begin{itemize}
\item  {\sf point module} if  it is cyclic, $1$-critical, and $H_M(t)=(1-t)^{-1}$; 
\item {\sf line module} if  it is cyclic, $2$-critical, and $H_M(t)=(1-t)^{-2}$.
\end{itemize}
We are also interested in modules of higher multiplicity: we call $M$ a
\begin{itemize}
\item  {\sf fat point module} if  it is $1$-critical, generated by $M_0$,  and $H_M(t)=n(1-t)^{-1}$ for some $n>1$. 
\end{itemize}
Point modules and fat point modules are important because, as objects in $\QGr(R)$, they are simple (or irreducible):  
all proper quotient modules of a 1-critical module are finite dimensional and therefore zero in $\QGr(R)$. 
The following result illustrates the relationship between finite dimensional simple modules and fat point modules.  

\begin{lemma}
\label{lem.old.Sm}
Let $V$ be a simple left $R$-module of dimension $n<\infty$. Let $\CC[z]$ be the polynomial ring generated by a degree-one indeterminate, $z$.
Let $V \otimes \CC[z]$ be the graded left $R$-module whose degree-$j$ component is $V \otimes z^j$  with $a \in R_i$ acting as $a(v \otimes z^j):=
(av) \otimes z^{i+j}$. Let $\pi:V \otimes \CC [z] \to V$ be the  $R$-module homomorphism $v \otimes z^j \mapsto v$. 
\begin{enumerate}
  \item 
  $V \otimes \CC[z]$ is a fat point  module of multiplicity $n$. 
\item If $M$ is a graded left $R$-module such that $M=M_{\ge 0}$ and
  $\psi: M \to V$ is a homomorphism in $\Mod(R)$, then there is a
  unique homomorphism $\widetilde{\psi}:M \to V \otimes \CC [z]$ in
  $\Gr(R)$ such that $\psi=\pi\psi$, namely
  $\widetilde{\psi}(m)=\psi(m) \otimes z^n$ for $m \in M_n$.
\end{enumerate} 
\end{lemma}

\subsection{Geometry in $\Projnc(R)$}
The ``non-commutative scheme'' $\Projnc(R)$ is defined implicitly by declaring the category of ``quasi-coherent sheaves'' on it is $\QGr(R)$,
$$
\qcoh(\Projnc(R)) \; := \;  \QGr(R).
$$
The isomorphism class of a (fat) point module in $\QGr(R)$ is called a (fat) point of $\Projnc(R)$. Likewise, the isomorphism class of a line module in $\QGr(R)$ 
is called a line in $\Projnc(R)$.

\subsubsection{Origin of the terminology} 
Let $\Bbbk$ be an algebraically closed field. Let
$R =\Bbbk[x_0,\ldots,x_n]$ be the commutative polynomial ring with its
standard grading, $\deg(x_j)=1$ for all $j$.  Then $\Proj(R)$ is
$\PP^n$, projective $n$-space, and there is a bijection between closed
points in $\PP^n$ and isomorphism classes of point modules for $R$: a
point module is isomorphic to $R/I$ for a unique ideal $I$, and $I$ is
generated by a codimension-1 subspace of $\CC x_0+\cdots +\CC x_n$;
conversely, if $I$ is such an ideal, then $R/I$ is a point module.
Under the equivalence
$\QGr(R) \stackrel{\sim}{\longrightarrow} \qcoh(\PP^n)$,
$M \mapsto \widetilde{M}$, the point module $R/I$ corresponds to the
skyscraper sheaf $\cO_p$ at the point $p \in \PP^n$ where $I$
vanishes.  Similarly, if $M$ is a line module for $R$, then
$M \cong R/I$ for an ideal $I$ that is generated by a codimension-2
subspace of $\CC x_0+\cdots +\CC x_n$ and the zero locus of $I$ is a
line in $\PP^n$, and this sets up a bijection between the lines in
$\PP^n$ and the isomorphism classes of line modules.  Indeed, there is
a bijection between linear subspaces of $\PP^n$ and isomorphism
classes of linear modules over the polynomial ring $R$.

\begin{theorem}
  \cite[Thm. 1.13]{LS93} Let $R = \Bbbk[x_0, \ldots, x_n]$ be a
  polynomial ring in $n+1$ variables, graded by setting $\deg(x_j)=1$
  for all $j$. Let $M$ be a finitely generated graded $R$-module. The
  following conditions on a graded $R$-module $M$ are equivalent:
\begin{enumerate}
  \item 
  $M$ is cyclic with Hilbert series $(1-t)^{-d}$; 
\item $M \cong R/R\ell^\perp$ for some codimension-$d$ subspace
  $\ell \subseteq R_1$ or, equivalently, for some $(d-1)$-dimensional
  linear subspace $\ell \subseteq \PP(R_1^*)$;
\item $M$ is a Cohen-Macaulay $R$-module having GK-dimension $d$ and
  multiplicity $1$.
\end{enumerate}
\end{theorem}

Thus, linear modules of GK-dimension $d$ correspond to linear subspaces of $\PP^n$ having dimension $d-1$.

\subsubsection{Points, fat points, and lines in $\Projnc(R)$}
\label{ssect.pts.lines}
For the non-commutative algebras $R$ in this paper, the points and lines in $\Projnc(R)$ are parametrized by genuine (commutative) varieties \cite{ATV1,ATV2}.

Let $R$ be any $\NN$-graded $\Bbbk$-algebra such that $R_0=\Bbbk$ and $R$ is generated by $R_1$ as a $\Bbbk$-algebra. Let $\PP(R_1^*)$
denote the projective space whose points are the 1-dimensional subspaces of $R_1^*$. 

For $V$ a linear subspace of $R_1^*$, define $V^\perp:=\{x \in R_1 \; | \; \xi(x)=0 \; \hbox{for all } \xi \in V\}$ and let
\begin{align*} 
\cP_R & \; :=\; \{p \in \PP(R_1^*) \; | \; R/Rp^\perp \, \hbox{ is a point module}\},
\\ 
\cL_R & \; := \; \{\hbox{lines $\ell$ in } \, \PP(R_1^*) \; | \; R/R\ell^\perp \, \hbox{ is a line module}\}.
\end{align*}
For the algebras in this paper there are moduli problems for which $\cP_R$ and $\cL_R$ are fine moduli spaces; 
\cite[Cor. 3.13]{ATV1} and \cite[Cor. 1.5]{ShV02}. We call $\cP_R$ and $\cL_R$ the {\sf point scheme} and {\sf line scheme} for $R$.   

It is easy to see that a line module $R/R\ell^\perp$ surjects onto a point module $R/Rp^\perp$ if and only if $p$ lies on the line $\ell$. 
Thus, the incidence relations between points and lines in $\Projnc(R)$ coincides with the incidence relations between {\it certain} points and lines in $\PP(R_1^*)$.
In such a situation the phrase {\it ``$p$ lies on $\ell$''} is a statement about points and lines in $\PP(R_1^*)$ and {\it also} a statement about points and lines in 
 $\Projnc(R)$. If a line module $R/R\ell^\perp$ surjects onto a fat point module $F$ in $\QGr(R)$ we say that the corresponding fat point lies on the line
  $\ell$ and understand this as a statement about
 $\Projnc(R)$.

 \begin{proposition}
 \cite{LS93}
 \label{prop.line.to.pt}
 The kernel of  a surjective homomorphism $\psi:M_\ell \to M_p$ in $\Gr(S)$ from a line module to a point module is isomorphic to a shifted line module $M_{\ell'}(-1)$.
 \end{proposition}
 \begin{proof}
 There are elements $u,v,w \in S_1$ for which there is a commutative diagram
 $$
 \xymatrix{ M_\ell \ar[d]_\psi \ar[r] & S/Su+Sv \ar[d]^{\psi'}
 \\
 M_p \ar[r] & S/Su+Sv +Sw
 }
 $$
 in which the horizontal arrows are isomorphisms and $\psi'$ is the obvious map.  The kernel of $\psi'$ is isomorphic to the submodule 
 $S\overline{w} =Su+Sv+Sw/Su+Sv$. Because $M_\ell$ is a critical Cohen-Macaulay module of GK-dimension 2 and multiplicity 1, and $M_p$ has GK-dimension 1,
 the kernel is a Cohen-Macaulay module of GK-dimension 2 and multiplicity 1. By \cite[Prop. 2.12]{LS93}, the kernel of $\psi'$ isomorphic to a shifted line module.
 \end{proof}
 
 The associated exact sequence $0 \to M_{\ell'}(-1) \to M_\ell \to M_p \to 0$ is the analogue of an exact sequence $0 \to M(\l') \to M(\l) \to L \to 0$ 
 in which $M(\l')$ and $M(\l)$ are Verma modules.

 \subsubsection{Noncommutative analogues of quadrics and $\PP^3$}

Let $S$ be  one of the algebras in \S\ref{ssect.S}. The Hilbert series of $S$ is $(1-t)^{-4}$, the same as that of the polynomial ring on 4 variables. Furthermore, $S$ 
has the ``same'' homological properties as that polynomial ring and, as a consequence, it is a domain \cite[Thm. 3.9]{ATV2}. For these reasons we think of 
$\Projnc(S)$ as a non-commutative analogue of $\PP^3$.

 If $\Omega$ is a homogeneous, degree-two, central element in $S$ we call $\Projnc(S/(\Omega))$
a {\it quadric hypersurface} in $\Projnc(S)$ and sometimes denote it by the symbols $\{\Omega=0\}$.
A  line module $S/S\ell^\perp$ is annihilated by $\Omega$ if and only if there is a surjective map
$S/(\Omega) \to S/S\ell^\perp$. If so we say that {\it ``the line $\ell$ lies on the quadric $\{\Omega=0\}$''} and we interpret this as a statement about the geometry of $\Projnc(S)$.

\subsection{The algebras $S$} 
\label{ssect.S}
The algebras of interest to us are the non-commutative $\CC$-algebras $S$ with generators $x_0,x_1,x_2,x_3$ subject to  the relations
\begin{align}
\label{eq.S.relations}
[x_0,x_1] & \; = \; \phantom{-} 0 & \{x_0,x_1\}  &\; =  \; 2x_0x_1 \; =  \; [x_2,x_3] \nonumber
\\
[x_0,x_2]& \; = \; \phantom{-}b^2\{x_1,x_3\}  & \{x_0,x_2\} &\; = \; [x_3,x_1]
\\
[x_0,x_3]& \; = \; -b^2\{x_1,x_2\}  &  \{x_0,x_3\} &\; = \; [x_1,x_2] \nonumber
\end{align}
where $\{x,x'\}=xx'+x'x$, $[x,x']=xx'-x'x$, and $b \in \CC-\{0,\pm i\}$.

The algebras $S$ occupy an interesting position between the non-degenerate $4$-dimensional Sklyanin algebras and the quantized enveloping 
algebras $U_q(\fsl_2)$. We now introduce these algebras and, in \Cref{prop.Uqsl2} below, describe their relation to $S$.

\subsubsection{$S$ is a degenerate Sklyanin algebra} 
A {\sf non-degenerate Sklyanin algebra} is a $\CC$-algebra $S(\alpha,\beta,\gamma)$ with generators $x_0,x_1,x_2,x_3$ subject to the relations
\begin{align}
\label{eq.S.nondegenerate}
[x_0,x_1] & \; = \; \alpha \{x_2,x_3\} & \{x_0,x_1\}  &\; =  \; [x_2,x_3] \nonumber
\\
[x_0,x_2] & \; = \; \beta \{x_1,x_3\}  & \{x_0,x_2\} &\; = \; [x_3,x_1]
\\
[x_0,x_3]& \; = \; \gamma \{x_1,x_2\}  &  \{x_0,x_3\} &\; = \; [x_1,x_2] \nonumber
\end{align}
where $\alpha,\beta,\gamma\in \CC$ are such that $\alpha + \beta + \gamma + \alpha\beta\gamma = 0$, and further satisfy the non-degeneracy condition
\begin{equation}
\label{eq:nondeg}
\{\alpha,\beta,\gamma\} \cap \{0,1,-1\} = \varnothing.
\end{equation}
With this notation, $S=S(0,b^2,-b^2)$, and is degenerate. 

\begin{quote}
For the rest of the paper, $S$ will denote $S(0,b^2,-b^2)$ and $S(\alpha,\beta,\gamma)$ will denote a non-degenerate Sklyanin algebra.
\end{quote}

The non-commutative space $\Projnc(S(\alpha,\beta,\gamma))$ is well understood. Its point scheme was computed in \cite{SS92}, its lines and the incidence 
relations between its points and lines  were determined in \cite{LS93}, and its fat points and the incidence relations between fat points and lines were 
determined in \cite{SSJ93}. A short account of these and related results can be found in the  survey article \cite{S94}.
In this paper we carry out the same computations for $S$ and compare them to what has been obtained for non-degenerate $S(\alpha,\beta,\gamma)$. This is the subject of \S\ref{sec.Sklyanin}.

\subsubsection{$S$ as a homogenization of $U_q(\fsl(2,\CC))$}
The quantized enveloping algebra $U_q(\fsl_2)$ is the $\CC$-algebra with generators $e,f,k^\pm$ subject to the relations
\begin{equation}
 \label{eq.Uq.relations}
 ke=q^2ek, \quad kf=q^{-2}fk,  \quad 
 kk^{-1}  =  k^{-1} k \; = 1,
 \quad \hbox{and} \quad
 [e, f] \; = \; \frac{k-k^{-1}}{q-q^{-1}},
 \end{equation}
 where $q \neq 0, \pm 1, \pm i$.  

 The representation theory of $U_q(\fsl_2)$ is the subject of books by Brown-Goodearl \cite{BG02}, Jantzen \cite{J96}, Kassel \cite{K95}, Klimyk-Schm\"udgen \cite{KS97}, and others. 
 \footnote{A slightly different algebra was studied in the 1985  paper \cite{Jimbo85} by Jimbo and in the 1988 paper \cite{Lusz88} by Lusztig: they replace the last of the above relations by $ [e, f] =(k-k^{-1})/(q^2-q^{-2})$. In his 1990 paper \cite{Lusz90}, Lusztig replaced that relation by the one in (\ref{eq.Uq.relations}) and that seems to have become the ``official'' quantized enveloping algebra of $\fsl(2,\CC)$ used by subsequent authors. We call the algebra studied in \cite{Jimbo85} and \cite{Lusz88} the ``unofficial'' quantized enveloping algebra of $\fsl(2,\CC)$. That unofficial version is a quotient of the algebra $S$ in \Cref{prop.Uqsl2}.}

Before showing that $S$ is a homogenization of $U_q(\fsl_2)$, we introduce notation that will be used throughout the paper:
\begin{align}
\label{eq.notation}
q &= \frac{1-ib}{1+ib}, & E &= \frac{i}{2}(1 - ib)(x_2+ix_3), & K\phantom{'} &= x_0 + bx_1, \\
\kappa &= \frac{1}{q^{-1}-q}, & F&= \frac{i}{2}(1 + ib)(x_2-ix_3), & K' &= x_0- bx_1. \nonumber
\end{align}

\begin{proposition}
\label{prop.Uqsl2}
The algebra $S$ is the $\CC$-algebra generated by $E,F,K,K'$ modulo the relations
\begin{align}
\label{eq.S.EFKrelations}
KE &= qEK, & KF &= q^{-1}FK, &  KK' &= K'K, \\
K'E &= q^{-1}EK' &  K'F &= qFK', & [E, F] &=  \frac{K^2-K'\,^2}{q-q^{-1}}. \nonumber
 \end{align}
Further, $S[(KK')^{-1}]_0$ is isomorphic to $U_q(\fsl_2)$ via
\begin{align}
\label{eq.Uqsl2.iso}
EK^{-1} &\mapsto \sqrt{q}e, & F(K')^{-1} &\mapsto \sqrt{q} f, & K(K')^{-1} &\mapsto k,
\end{align}
where $\sqrt{q}$ is a fixed square root of $q$.
\end{proposition}

\begin{proof}
A few tedious but straightforward calculations show that $E,F,K,K'$ satisfy the relations in (\ref{eq.S.EFKrelations}). For example,  $KE=qEK$ because
 \begin{align*}
 (1+ib)KE-(1-ib)EK & \; = \; [K,E]+ib\{K,E\}
 \\
 & \; = \;\frac{i}{2}(1 - ib)\bigg( [x_0,x_2] \, +\, i[x_0,x_3] \, +\, b[x_1,x_2]  \, \, +\, ib[x_1,x_3] 
 \\
 &  \qquad \, +\, ib \{x_0,x_2\} \, -\, b \{x_0,x_3\} \, +\,  ib^2\{x_1,x_2\}  \, - \,  b^2 \{x_1,x_3\} \bigg)
  \\
 & \; = \;  \frac{i}{2}(1 - ib)\bigg( [x_0,x_2] \, - \, b^2\{x_3,x_1\} +i[x_0,x_3]  \, +\,  ib^2\{x_1,x_2\}
  \\
 & \qquad  \, +\, ib \{x_0,x_2\}  \, -\, ib[x_3,x_1] \, -\, b \{x_0,x_3\}   \, +\, b[x_1,x_2]  \bigg)
 \\
 & \; = \; 0.
 \end{align*} 
Similar calculations show $KF=q^{-1}FK$, $K'E=q^{-1}EK'$ and $K'F=qFK'$.

Since $K^2-K'^2=4bx_0x_1=2b\{x_0,x_1\}$ and $\frac{i}{2}(1 - ib) \cdot \frac{i}{2}(1 + ib) = -\frac{1+b^2}{4}$, we have
$$
- \,  \frac{4}{1+b^2} [E,F] \, + \, ib^{-1} \big(K^2-K'^2 \big)  \; = \; 2i[x_3,x_2]  \,+\, 2i  \{x_0,x_1\}  \; = \; 0.
$$
However,
  $$
 q-q^{-1}  \; = \; \frac{1-ib}{1+ib}   \; - \; \frac{1+ib}{1-ib}  \; =  \;  - \,   \frac{4ib}{1+b^2} \, ,
 $$  
 so 
 $$
 [E, F] \; = \;  -  \, \frac{1+b^2}{4ib}  (K^2-K'\,^2 ) \; = \; \frac{K^2-K'\,^2}{q-q^{-1}}.
 $$

For the second part of the proposition, it is clear that $S[(KK')^{-1}]_0$ is generated by 
\begin{align*}
e &:= \frac{1}{\sqrt{q}}EK^{-1}, & f &:= \frac{1}{\sqrt{q}} F(K')^{-1}, & k &:=K(K')^{-1}, & \mbox { and } \qquad k^{-1}&.
\end{align*}
Similar straightforward calculations then show that these elements satisfy the relations in (\ref{eq.Uq.relations}). 
Hence $S[(KK')^{-1}]_0 \cong U_q(\fsl_2)$.
\end{proof}

Since $S$ is an Artin-Schelter regular algebra of global dimension 4 and has Hilbert series $(1-t)^{-4}$ we think of it as a homogeneous coordinate ring of a non-commutative analogue of $\PP^3$. Since $SK=KS$ and $SK'=K'S$ we think of $S/(K)$ and $S/(K')$ as homogeneous coordinate rings of non-commutative analogues of $\PP^2$. 

Further, we think of $S[(KK')^{-1}]_0$, i.e., $U_q(\fsl_2)$, as the
coordinate ring of the non-commutative affine scheme that is the
``open complement'' of the ``union'' of the ``hyperplanes'' $\{K=0\}$
and $\{K'=0\}$. These ``hyperplanes'' are effective divisors in the
sense of Van den Bergh \cite[\S3.6]{VdB-blowup}.  From this
perspective, $U_q(\fsl_2)$ can be considered an ``affine piece'' of
$S$. As explained in \S\ref{ssect.localize}, this point of view can be
formalized in terms of an adjoint pair of functors $j^* \dashv j_*$.

The left adjoint $j^*:\QGr(S)\to \Mod(S[(KK')^{-1}]_0)$ sends a graded $S$-module, $X$, viewed as an object in $\QGr(S)$ to $X[(KK')^{-1}]_0\in \Mod(S[(KK')^{-1}]_0)$.
The action of $j^*$ on a morphism $f:M\to N$ in $\QGr(S)$ is defined by first choosing a lift of $f$ to a morphism $\phi$ $\Gr(S)$ and then applying the
localization functor $X\mapsto X[(KK')^{-1}]_0$ to $\phi$.

\section{Point and line modules for $D$, a Zhang twist of $S$}
\label{sec.D}
  In this section, we replace $S$ by a Zhang twist of itself \cite{Z-twist}. The appropriate Zhang twist is an algebra $D$ that has has a central element $z \in D_1$ such that $D/(z)$ a 3-dimensional Artin-Schelter regular algebra. In the terminology of \cite{LBSvdB},
this makes $D$ a {\it central extension of $D/(z)$}. We use the results in \cite{LBSvdB} to determine the point and line modules for $D$. 
The point and line modules for $D/(z)$ are already understood due to \cite{ATV1} and \cite{ATV2}. 

In \S\ref{sec.S} we use  Zhang's fundamental equivalence $\Gr(D)\equiv  \Gr(S)$ \cite{Z-twist} to transfer the results about  the point and line modules for $D$ to $S$. 

\subsection{The Zhang twist \cite{Z-twist}}
Let $S$ be a graded $\Bbbk$-algebra and $\phi$ a degree-preserving  $\Bbbk$-algebra automorphism of $S$. Define $D$ to be the  $\Bbbk$-algebra that is equal to $S$ as a graded  $\Bbbk$-vector space, but endowed with a new multiplication 
 $$
 c*d\; := \; \phi^n(c) d
 $$
 for $c\in D=S$ and $d \in D_n=S_n$. We call $D$ a {\sf Zhang twist} of $S$. In \cite{Z-twist}, Zhang showed that there is an equivalence of categories $\Phi:\Gr(S) \to \Gr(D)$ defined as follows:
 if $M$ is a graded left $S$-module, then $\Phi M$ is $M$ as a graded  $\Bbbk$-vector space, but endowed with the $D$-action
  $$
 c*m\; := \; \phi^n(c) m
 $$
for $c \in D=S$ and $m \in (\Phi M)_n=M_n$.  

On morphisms $\Phi$ is the ``identity'': if $f \in {\rm Hom}_{\Gr(S)}(M,M')$, then
 $\Phi(f)=f$ considered now as a morphism $\Phi M \to \Phi M'$. Note that $f$ {\it is} a morphism of graded $D$-modules because if $c \in D$ and $m \in M_n$, then $f(c*m)=f(\phi^n(c)m)=\phi^n(c)f(m)=c*f(m)$.   

We use the following graded algebra automorphism $\phi:S \to S$ defined by 
\begin{equation}
\label{eq:phi}
\phi(s):=K's(K')^{-1}.
\end{equation}
This is a homomorphism because $K'S=SK'$, and is an automorphism because $S$ is a 4-dimensional AS-regular algebra and therefore a domain \cite[Thm. 3.9]{ATV2}. 

 \begin{proposition}
 \label{prop.defn.D}
Let $D$ be the Zhang twist of $S$  with respect to $\phi$ (\ref{eq:phi}). Then $D$ is isomorphic to   $\CC\inner{E,F,K,K'}$ modulo the relations
\begin{align*}
[K',E] & \; = \; [K',F] \; = \; [K',K]  \; = \; 0, 
\\
\; KE \; =\; q^2 EK, \qquad \quad KF & \; = \; q^{-2}FK,  \qquad \quad
q EF- q^{-1} FE \; = \;  \frac{K^2-K'^{2}}{q-q^{-1}}. 
\end{align*}
Further, $K'$ belongs to the center of $D$.  
\end{proposition}
\begin{proof}
 Since $\phi(E)= q^{-1} E$,  $\phi(F)= q  F$, $\phi(K)=K$, and $\phi(K')=K'$, 
 \begin{align*}
 K'*E & \; = \; \phi(K')E \; = \; K'E \; = \;  q^{-1} EK' \; = \;  \phi(E)K' \; = \;  E*K',
 \\
 K'*F & \; = \; \phi(K')F\; = \; K'F \; = \;  q FK' \; = \;  \phi(F)K' \; = \;  F*K',
 \\
  K*E & \; = \; \phi(K)E \; = \; KE \; = \;  q EK \; = \;  q^2 \phi(E)K \; = \;  q^2 E*K,
 \\
 K*F & \; = \; \phi(K)F\; = \; KF \; = \;  q^{-1} FK \; = \;  q^{-2} \phi(F)K \; = \;  q^{-2}F*K,
 \end{align*}
 and 
 $$
q E*F -q^{-1} F*E  = EF- FE = \frac{K^2-K'^{2}}{q-q^{-1}}. 
  $$
By the very definition of $\phi$, $K'$ belongs to the center of $D$.  
 \end{proof} 
 
 \begin{corollary}
 \label{cor.A.relns}
 Let $A$ be $\CC \inner{E,F,K}$ modulo the relations
 $$
 KE=q^2EK, \qquad  KF=q^{-2}FK, \qquad  q EF -q^{-1} FE = \frac{K^2}{q-q^{-1}}. 
 $$
 Then  $A\cong D/(K')$ and $D$ is a central extension of $A$ in the sense of \cite[Defn. 3.1.1]{LBSvdB}. 
 \end{corollary}

\subsection{Applying the results in \cite{LBSvdB}}
\label{ssect.A.points}

In the notation of \cite{LBSvdB},
our $(E,F,K,K')$ is their $(x_1,x_2,x_3,z)$. Following the notation in \cite[Eq. (3.1) and \S4.2]{LBSvdB},
if $A=\CC \inner{x_1,x_2,x_3}/(f_1,f_2,f_3)$, then the defining relations for the central extension $D$ of $A$ can be written as\footnote{In the notation of \cite[Thm. 3.1.3]{LBSvdB}, $\gamma_j=0$ for all $j$.}
\begin{align*}
zx_i-x_iz & \; = 0, \; \quad j=1,2,3, \qquad \hbox{and}
\\
g_j \; := \; f_j + z l_j+\a_jz^2 & \; = 0, \; \quad j=1,2,3
\end{align*}
for some $l_j \in A_1$ and $\a_j \in \CC$. 
For our $D$, 
\begin{equation}
\label{defn.bf.f}
{\bf f} :=  \begin{pmatrix}   f_1 \\ f_2 \\ f_3  \end{pmatrix} 
= \begin{pmatrix}   -q^{3}KF+qFK \\ q^{-3}KE-q^{-1}EK \\  qEF-q^{-1}FE + \kappa K^2  \end{pmatrix}, \quad
{\bf l} := \begin{pmatrix}   l_1 \\ l_2 \\ l_3  \end{pmatrix} =\begin{pmatrix}   0 \\ 0 \\ 0  \end{pmatrix}, \quad
{\boldsymbol \a} := \begin{pmatrix}   \a_1 \\ \a_2 \\ \a_3  \end{pmatrix} =  \begin{pmatrix}   0 \\ 0 \\ -\kappa  \end{pmatrix} .
\end{equation} 
Thus
\begin{equation}
\label{defn.bf.g}
 {\bf g} \; :=\; \begin{pmatrix} g_1 \\ g_2 \\ g_3  \end{pmatrix} \; = \; 
  \begin{pmatrix}   -q^{3}KF+qFK \\ q^{-3}KE-q^{-1}EK \\  qEF-q^{-1}FE + \kappa K^2 -\kappa K'^2  \end{pmatrix}.
\end{equation}
The relations of $A$ are said to be in {\sf standard form} \cite[p.34]{ATV1} if (in the notation of \cite[p.181]{LBSvdB}) there is a $3 \times 3$ matrix $M$, 
and a matrix $Q \in \GL(3)$, such that ${\bf f}=M{\bf x}$ and ${\bf x}^\sT M = (Q{\bf f})^\sT$, where ${\bf f}^\sT=(f_1,f_2,f_3)$ and $A$ is generated as an algebra by the entries of the column vector ${\bf x}$.

\begin{proposition}
\label{prop.Q}
The relations ${\bf f}$ for $A$ in (\ref{defn.bf.f}) are in standard form, where 
\begin{align}
{\bf x} &= (E,F,K)^\sT, \\
Q &= {\sf diag}(q^{-4},q^4,1), \nonumber
\end{align}
and
\begin{equation}
\label{defn.M}
M =
\begin{pmatrix}
  0 & -q^3K    & qF                    \\
  q^{-3}K & 0     &  -q^{-1}E          \\
  -q^{-1}F & qE &  \kappa K
\end{pmatrix}.
\end{equation}
\end{proposition}
\begin{proof}
It is easy to check that ${\bf f} = M {\bf x}$. On the other hand,  
$$
{\bf x}^\sT M \; = \; (E, \, F, \, K)   M \; = \; 
(q^{-3}FK-q^{-1}KF, \, -q^3EK+qKE, \, qEF-q^{-1}FE + \kappa K^2),
$$
so
$$
({\bf x}^\sT M)^\sT \;=\; \begin{pmatrix} q^{-3}FK-q^{-1}KF \\ -q^3EK+qKE \\ qEF-q^{-1}FE + \kappa K^2 \end{pmatrix}
\; = \; 
\begin{pmatrix} 
q^{-4} & 0 & 0   \\
  0    &  q^4 & 0  \\
  0    & 0 & 1 
\end{pmatrix}{\bf f}.
$$
Thus ${\bf x}^\sT M = (Q{\bf f})^\sT$ as claimed. 
\end{proof}

We use $(E,F,K)$ as homogeneous coordinate functions on the plane $\PP(A_1^*) \cong \PP^2$ and identify this
plane with the hyperplane $K'=0$ in $\PP(D_1^*)$.

\begin{proposition}
\label{prop.A.points}
The point scheme $(\cP_A,\s_A)$ for $A$ is the cubic divisor consisting of the line $K=0$ and the conic
$\kappa^ 2K^2+ EF=0$. The line meets the conic at the points $(1,0,0)$ and $(0,1,0)$.  
\begin{enumerate}
  \item 
   If  $(\xi_1 ,\xi_2 ,\xi_3 )$ lies on the conic  $\kappa^ 2 K^2+EF=0$,
then $\s_A(\xi_1 ,\xi_2 , \xi_3  ) = (q^2\xi_1 , \,  q^{-2} \xi_2 , \, \xi_3 )$.
  \item 
If $(\xi_1 ,\xi_2 ,\xi_3 )$ lies on the line $K=0$; i.e., $\xi_3 =0$, then 
$\s_A(\xi_1 ,\xi_2 , 0 ) = (q\xi_1 ,q^{-1}\xi_2 , 0 )$. 
\end{enumerate}
\end{proposition}
\begin{proof}
By \cite{ATV1}, the subscheme of $\PP(A_1^*)$ parametrizing the left point modules for $A$ is given by the equation
$$
\det \begin{pmatrix}
  0 & -q^2K    & F                    \\
  q^{-2}K & 0     &  -E          \\
  -q^{-1}F & qE &  \kappa K
\end{pmatrix} =0.
$$
The vanishing locus of this determinant is the union  
of the line $K=0$ and the smooth conic $\kappa^ 2K^2+ EF=0$. The line meets the conic at the points $(1,0,0)$ and $(0,1,0)$.

We denote this cubic curve by $\cP_A$. The point module corresponding to a point $p \in \cP_A$ is $M_p:=A/Ap^\perp$ where 
$p^\perp$ is the subspace of $A_1$ consisting of the linear forms that vanish at $p$.  

If $M_p$ is a point module for $A$ so is $(M_p)_{\ge 1}(1)$. 
In keeping with the notation in \cite{LBSvdB}, we write $\s_A$ (in this proof we will use $\s$ for brevity)  for the automorphism of $\cP_A$ such that 
\begin{equation}
\label{defn.sigma}
M_{\s^{-1}(p)} \cong  (M_p)_{\ge 1}(1).
\end{equation}

To determine $\s$ explicitly, let $p \in \cP_A$ and suppose that $p=(\xi_1',\xi_2',\xi_3')$ and $\s^{-1}(p)=(\xi_1,\xi_2,\xi_3)$
with respect to the homogeneous coordinates $(E,F,K)$. Then $M_p$ has a homogeneous basis $e_0,e_1,\ldots$ 
where $\deg(e_n)=n$ and 
$$
E  e_0=\xi_1' e_1,  \quad F  e_0=\xi_2' e_1,  \quad K  e_0=\xi_3' e_1,
$$
and 
$$
E  e_1=\xi_1 e_2,  \quad F  e_1=\xi_2 e_2,  \quad K  e_1=\xi_3 e_2.
$$
Since $KE-q^2EK=0$ in $A$, 
$(KE-q^2EK)  e_0=0$; i.e.,  $\xi_3\xi_1'-q^2\xi_1 \xi_3'=0$.  The other two relations for $A$ in \Cref{cor.A.relns}
imply $\xi_3 \xi_2'-q^{-2}\xi_2 \xi_3'=0$ and $q \xi_1 \xi_2'-q^{-1}\xi_2 \xi_1' +\kappa \xi_3 \xi_3' =0$. 
These equalities can be expressed as the single equality
$$
\begin{pmatrix}
0 & \xi_3  & - q^{-2}\xi_2     \\
\xi_3  & 0 &  -q^2\xi_1         \\
 -q^{-1} \xi_2  & q \xi_1  & \kappa \xi_3 
\end{pmatrix}
\begin{pmatrix}
\xi_1' \\ \xi_2' \\ \xi_3'
\end{pmatrix} = 0.
$$
Since $\s(\xi_1 ,\xi_2 ,\xi_3 )=(\xi_1',\xi_2',\xi_3')$, we can now determine $\s$. 

If $(\xi_1 ,\xi_2 ,\xi_3 )$ lies on the line $K=0$; i.e., $\xi_3 =0$, then 
$$
\begin{pmatrix}
0 & 0 & - q^{-2}\xi_2     \\
0 & 0 &  -q^2\xi_1         \\
 -q^{-1} \xi_2  & q \xi_1  & 0
\end{pmatrix}
\begin{pmatrix}
\xi_1' \\ \xi_2' \\ \xi_3'
\end{pmatrix} = 0
$$
so $\s(\xi_1 ,\xi_2 , 0 ) = (q\xi_1 ,q^{-1}\xi_2 , 0 )$. If  $(\xi_1 ,\xi_2 ,\xi_3 )$ lies on the conic  $\kappa^ 2 K^2+EF=0$,
then 
$$
\begin{pmatrix}
0 & \xi_3  & - q^{-2}\xi_2     \\
\xi_3  & 0 &  -q^2\xi_1         \\
 -q^{-1} \xi_2  & q \xi_1  & \kappa \xi_3 
\end{pmatrix}
\begin{pmatrix}
q^2\xi_1  \\ q^{-2} \xi_2  \\ \xi_3 
\end{pmatrix} = 0
$$
so  $\s(\xi_1 ,\xi_2 , \xi_3  ) = (q^2\xi_1 , \,  q^{-2} \xi_2 , \, \xi_3 )$.
\end{proof}

The algebra $A$ is of Type $S_1'$ in the terminology of \cite[Prop. 4.13, p. 54]{ATV1}. 
See also \cite[p.187]{LBSvdB} where it is stated that $D$ is the unique central extension of $A$ that is not a 
polynomial extension, up to the notion of equivalence at \cite[\S3.1,p.180]{LBSvdB}.

In the next result, which is similar to \Cref{prop.A.points}, we use $(E,F,K')$ as homogeneous coordinate
functions on the plane $K=0$ in $\PP(D_1^*)$.  

\begin{proposition}
\label{prop.A'.points}
Let $A'=D/(K)$.
The point scheme $(\cP_{A'},\s_{A'})$ for $A'$ is the cubic divisor on the plane $K=0$ consisting of the line $K'=0$ and 
the smooth conic $EF+\kappa^ 2K'^2=0$. The line meets the conic at the points $(1,0,0)$ and $(0,1,0)$.   
\begin{enumerate}
  \item 
   If  $(\xi_1 ,\xi_2 ,\xi_4 )$ lies on the conic  $\kappa^ 2 K'^2+EF=0$,
then $\s_{A'}(\xi_1 ,\xi_2 , \xi_4  ) = (\xi_1 , \,  \xi_2 , \, \xi_ 4)$.
  \item 
If $(\xi_1 ,\xi_2 ,\xi_4 )$ lies on the line $K=0$; i.e., $\xi_4 =0$, then 
$\s_{A'}(\xi_1 ,\xi_2 , 0 ) = (q\xi_2 ,q^{-1}\xi_1 , 0 )$. 
\end{enumerate}
\end{proposition}
\begin{proof}
Since $[E,K']=[F,K']=qEF-q^{-1}FE-\kappa K'^2=0$ are defining relations for $A'=\CC[E,F,K']$,  the left point modules for $A'$
are naturally parametrized by the scheme-theoretic zero locus of 
$$
\det \begin{pmatrix}
  0 & K'    & -F                    \\
  K' & 0     &  -E          \\
  -q^{-1}F & qE &  -\kappa K'
\end{pmatrix}
$$
in  $\PP(A_1^{'*})$, 
namely the union  
of the line $K'=0$ and the smooth conic $\kappa^ 2{K'}^2+ EF=0$. 

We denote this cubic curve by $\cP_{A'}$ and define $\s_{A'} :\cP_{A'} \to \cP_{A'}$ (in this proof we will use $\s$ for brevity)  by the requirement that 
$M_{\s^{-1}(p)} \cong  (M_p)_{\ge 1}(1)$ for all $p \in \cP_{A'}$. Calculations like those in \Cref{prop.A.points} show that $\s$ is the identity on the conic and is given by  $\s(\xi_1,\xi_2,0)=  (q\xi_2 ,q^{-1}\xi_1 , 0 )$ on the line $K'=0$.  
\end{proof}

\subsection{The  point scheme for $D$}
\label{ssect.D.points}

By \cite[Thm. 4.2.2]{LBSvdB} the point scheme $(\cP_D,\s_D)$ for $D$ exists. That result also gives an  explicit description of $\cP_D$. It is also pointed out there that the restriction of $\s_D$ to $\cP_D-\cP_A$ is the identity. 

\emph{Warning}:
The $g_1,g_2,g_3$ in (\ref{defn.bf.g}) belong to the tensor algebra $T(D_1)$.
 The $g_1,g_2,g_3$ in the next result are the images of the $g_1,g_2,g_3$ in (\ref{defn.bf.g}) in the polynomial ring
 generated by the indeterminates $E,F,K,K'$.  

\begin{proposition}
\label{prop.D.pts}
\cite[Lem. 4.2.1 and Thm. 4.2.2]{LBSvdB}
Let $x_1=E$, $x_2=F$, and $x_3=K$. 
The  equations for $\cP_D$ are
\begin{enumerate}
  \item
  $g_1=g_2=g_3=0$ on $\cP_D \cap \{K' \ne 0\}$ where
  $$
  \begin{pmatrix} g_1 \\ g_2 \\ g_3  \end{pmatrix} \; = \; 
  \begin{pmatrix} (q-q^3)FK \\ (q^{-3}-q^{-1})EK \\ -\kappa^ {-1} EF + \kappa K^2 - \kappa K'^2 \end{pmatrix} \; = \;
  \kappa^ {-1}\begin{pmatrix} q^2 FK \\ q^{-2}EK \\ -EF + \kappa^ 2(K^2 - K'^2) \end{pmatrix}  ,
  $$
  and  
  \item
  $K'g_1=K'g_2=K'g_3=h_i=0$ on $\cP_D \cap \{x_i \ne 0\}$ where
  $$
  \begin{pmatrix} h_1 \\ h_2 \\ h_3  \end{pmatrix} \; = \; 
 \kappa^ {-1} K \begin{pmatrix} E (EF + \kappa^ 2 K^2 -\kappa^ 2 q^2 K'^2 )
\\
F (EF  + \kappa^ 2 K^2-\kappa^ 2 q^{-2} K'^2 )
\\
K (EF+\kappa^ 2 K^2 -\kappa^ 2  K'^2) \end{pmatrix}.
  $$
\end{enumerate}
 \end{proposition} 
\begin{proof}
 The polynomials $h_1$, $h_2$, and $h_3$ are defined in \cite[Lem. 4.2.1]{LBSvdB}. 

Denote the columns of $M$ by $M_1$, $M_2$, $M_3$, so that $M=[M_1\, M_2\, M_3]$, and note that
\begin{equation*}
\det(M) = (\kappa K^2 +\kappa^ {-1}EF)K.
\end{equation*}

In this case, since ${\bf l}=0$, the definitions reduce to
\begin{align*}
h_1 & \; = \; E\det(M) + z^2 \det[{\boldsymbol \a}\, M_2 \, M_3] 
\\
h_2 & \; = \; F\det(M) + z^2 \det[M_1\, {\boldsymbol \a}\, M_3] 
\\
h_3 & \; = \; K\det(M) + z^2 \det[M_1 \, M_2\, {\boldsymbol \a}].
\end{align*}
Since  ${\boldsymbol \alpha}=(0,0,-\kappa)^\sT$, 
\begin{align*}
h_1 & \; = \; EK (\kappa K^2+\kappa^ {-1}EF) -\kappa q^2 K'^2 EK,
\\
h_2& \; = \; FK (\kappa K^2+\kappa^ {-1}EF) -\kappa q^{-2} K'^2 FK  
 \quad \hbox{and}
\\
h_3 & \; = \; K^2 (\kappa K^2+\kappa^ {-1}EF) -\kappa  K'^2 K^2. 
\end{align*} 
Hence the result.
\end{proof}

\begin{theorem}
\label{thm.PD}
The point scheme $\cP_D$ is reduced and is the union of 
\begin{enumerate}
  \item 
  the conics $EF+\kappa^ 2K^2=K'=0$ and $EF+\kappa^ 2K'^2=K=0$,
  \item 
  the line $K=K'=0$, and
  \item 
  the points $(0,0,1,\pm 1)$.
\end{enumerate}
Let $p \in \cP_D$. 
\begin{enumerate}
  \item[(4)] 
 If $p=(\xi_1,\xi_2,\xi_3,0)$ is on the conic $EF+\kappa^ 2K^2=K'=0$, then $\s_D(p)= (q^2\xi_1,q^{-2}\xi_2,\xi_3,0)$.
  \item[(5)] 
  If $p=(\xi_1,\xi_2,0,\xi_4)$ is on the conic $EF+\kappa^ 2K'^2=K=0$,  then $\s_D(p)=p$.
  \item[(6)] 
   If $p=(\xi_1,\xi_2,0,0)$ is on the line $K=K'=0$, then $\s_D(p)= (q\xi_1,q^{-1}\xi_2,0,0)$.
  \item[(7)] 
If $p=(0,0,1,\pm 1)$, then $\s_D(p)=p$.
\end{enumerate} 
\end{theorem}
\begin{proof}
By \cite[Thm. 4.2.2]{LBSvdB}, $(\cP_D)_{\rm red}=(\cP_A)_{\rm red} \cup V(g_1,g_2,g_3)_{\rm red}$ where $V(g_1,g_2,g_3)$ is the 
scheme-theoretic zero locus of the ideal $(EK,FK, EF-\kappa^2(K'^2-K^2)$.  Certainly $\cP_A$ is reduced. 
Straightforward computations on the open affine pieces $E \ne 0$, $F \ne 0$, $K\ne 0$, and $K' \ne 0$, show that $V(g_1,g_2,g_3)$ is reduced. 
Hence $\cP_D$ is reduced.

If $p=(0,0,1,\pm 1)$, then $M_p=D/Dp^\perp = D/DE+DF + D(K \mp K')$. But $DE+DF + D(K \mp K')$ is a two-sided ideal
and the quotient by it is the polynomial ring in one variable. Hence $\s_D(p)=p$. 
\end{proof}

\subsection{The line modules for $D$}
\label{ssect.D.lines}
We now use the results in \cite[\S5]{LBSvdB} to characterize the line modules for $D$.
Recall from \S\ref{ssect.pts.lines} that
$$
\cL_D \; = \; \{\hbox{lines $\ell$ in } \, \PP(D_1^*) \; | \; D/D\ell^\perp \, \hbox{ is a line module}\}.
$$
For each point $p \in \cP_A$,  let
$$
\cL_p \; := \; \{\ell \in \cL_D \; | \; p \in \ell\}.
$$
Then
$$
\cL_D \; = \; \{\hbox{lines on the plane $K'=0$}\} \; \cup \;  \bigcup_{p \in \cP_A} \cL_p.
$$

\begin{proposition}
Let $M$ be a line module for $D$.
\begin{enumerate}
  \item 
There is a unique line $\ell$ in $\PP(D_1^*)$ such that $M \cong D/D\ell^\perp$.
  \item 
If $K'M=0$, then $\ell \subseteq \{K'=0\}$ and $M$ is a line module for $A=D/(K')$.   
  \item 
The line modules for $A$ are, up to isomorphism, $A/A\ell^\perp$ where $\ell    \subseteq \{K'=0\}$.
  \item
If $K'M\ne 0$, then $M/K'M$ is a point module for $A$ and is isomorphic to $A/Ap^\perp$ where $\{p\}=\ell \cap\{K'=0\}$. 
\end{enumerate}
\end{proposition}
\begin{proof}
(1)
This is a consequence of  \cite[Prop. 2.8]{LS93} which says that if 
$A$ is a noetherian, Auslander regular, graded  $\Bbbk$-algebra having Hilbert series $(1-t)^{-4}$,  
and is generated by $A_1$, and satisfies the Cohen-Macaulay property, then there is a bijection 
 $$
  \{\hbox{lines } u=v=0 \hbox{  in $\PP(A_1^*)$} \; | \; \hbox{there is a rank 2 relation $a \otimes u-b \otimes v$}\} 
\; \longleftrightarrow \; \frac{A}{Au+Av}
 $$
between certain lines in $\PP(A_1^*)$ and the set of isomorphism classes of line modules for $A$.

(2) 
This is obvious.

(3) 
Since $A$ is a 3-dimensional Artin-Schelter regular algebra, by \cite{ATV1}, the isoclasses of the line modules for $A$ are the modules
$A/A\ell^\perp$ where $\ell$ ranges over all lines in $\PP(A_1^*)=\{K'=0\}$. 

(4)
See the discussion at \cite[p.204]{LBSvdB}.
\end{proof}

 \subsubsection{The quadrics $Q_p$}
 \label{sssect.quadrics.Qp}
 By \cite[Thm. 5.1.6]{LBSvdB}, if $p \in \cP_A$ there is a quadric $Q_p$ containing $p$ such that 
$$
\cL_p=\bigg\{\hbox{lines $\ell \subseteq \{K'=0\}$ such that $p \in \ell$}\bigg\} \; \cup \;
\bigg\{\hbox{lines $\ell \subseteq Q_p$ such that $p \in \ell$}\bigg\}.
$$
By \cite[Prop. 5.1.7]{LBSvdB}, if $\s_D(p)=(\zeta_1,\zeta_2,\zeta_3,0)$, then $Q_p$ is given by the equation ${\boldsymbol \zeta}^\sT Q {\bf g}=0$ where ${\boldsymbol \zeta}=(\zeta_1,\zeta_2,\zeta_3)^\sT$, 
$Q$ is the matrix in \Cref{prop.Q},  ${\bf g} =(g_1,g_2,g_3)^\sT = {\bf f} + {\boldsymbol \alpha}K'^2$, ${\bf f}$ is the image in the polynomial ring $\CC[E,F,K]$ of the 
column vector ${\bf f}$ introduced in the proof of  \Cref{prop.Q}, and ${\boldsymbol \alpha} = (0,0,-\kappa)^\sT$.  Thus, $Q_p$ is given by the equation
$$
(\zeta_1,\,\zeta_2,\,\zeta_3) 
\begin{pmatrix} 
q^{-4} & 0 & 0   \\
  0    &  q^4 & 0  \\
  0    & 0 & 1 
\end{pmatrix}
  \kappa^ {-1}\begin{pmatrix} q^2 FK \\ q^{-2}EK \\ -EF + \kappa^ 2K^2 - \kappa^ 2K'^2 \end{pmatrix} 
\; = \; 0
$$
or, equivalently, by 
$$
\kappa^ {-1} (\zeta_1,\,\zeta_2,\,\zeta_3) 
\begin{pmatrix} q^{-2} FK \\ q^{2}EK \\ -EF + \kappa^ 2K^2 - \kappa^ 2K'^2 \end{pmatrix}  
\; = \; 0.
$$

The next result determines $\cL_p$ for each $p \in \cP_A$.
We take coordinates  with respect to the coordinate functions $(E,F,K,K')$.

\begin{proposition}
\label{prop.lines.D.1}
Suppose $p = (\xi_1,\xi_2,\xi_3,0) \in \cP_A$.
\begin{enumerate}
\item{}
If $p=(1,0,0,0)$, then
 $\cL_p=\{\hbox{lines $\ell \subseteq \{F=0\} \cup \{K=0\} \cup \{K'=0\}$ such that $p \in \ell$}\}$.
\item{}
If $p=(0,1,0,0)$, then 
$\cL_p=\{\hbox{lines $\ell \subseteq \{E=0\} \cup \{K=0\} \cup \{K'=0\}$ such that $p \in \ell$}\}$.
  \item 
  If $\xi_3=0$ and $\xi_1\xi_2\ne 0$, then 
 $\cL_p=\{\hbox{lines $\ell \subseteq \{K=0\} \cup \{K'=0\}$ such that $p \in \ell$}\}$.
  \item 
   If $\xi_3\ne 0$, then $Q_p$ is the cone with vertex $p$ given by the equation 
   $$
   \xi_1  FK + \xi_2 EK + \xi_3(-EF + \kappa^ 2 K^2 - \kappa^ 2 K'^2)=0,
   $$ 
and     $\cL_p=\{\hbox{lines $\ell \subseteq Q_p \cup \{K'=0\}$ such that $p \in \ell$}\}$.
\end{enumerate}
\end{proposition}
\begin{proof}
(1)--(3)
Suppose $\xi_3=0$. Then $p$ is on the line $\{K=K'=0\}$ whence $\s_D(p) = (\xi_1,\xi_2,0,0)$ and $Q_p$ is given by
the equation
$$
\kappa^ {-1} \big( \xi_1q^{-2} F +\xi_2 q^2 E\big)K =0.
 $$
 Thus, $\ell$ either  lies on the pair of planes $\{KK'=0\}$ or the plane $\{\xi_1q^{-2} F +\xi_2 q^2 E=0\}$. 
 Suppose $\ell \not\subseteq \{KK'=0\}$. We are assuming that $q^4+1 \ne 0$ so, if $p$ is neither 
 $(1,0,0,0)$ nor $(0,1,0,0)$, then  $\xi_1q^{-2} F +\xi_2 q^2 E$ does not vanish at $p$ whence there are 
 no lines through $p$ that lie on the plane $\{\xi_1q^{-2} F +\xi_2 q^2 E=0\}$.
 
 Suppose $p=(1,0,0,0)$. Then  $Q_p=\{FK=0\}$ and $\cL_p$ consists of the lines through $p$ that are contained in $\{F=0\}
 \cup \{K=0\}$. Similarly, if  $p=(0,1,0,0)$, then $\cL_p$ consists of the lines through $p$ that are contained in $\{E=0\}
 \cup \{K=0\}$. 
  
 (4) Suppose $\xi_3 \ne 0$. Then $p$ lies on the conic
 $\kappa^ 2K^2+EF=K'=0$ so $\s_A(p) = (q^2\xi_1,q^{-2}\xi_2,\xi_3,0)$.
 Thus, $Q_p$ is given by the equation
 \begin{equation*}
   \kappa^ {-1} (q^2\xi_1,\, q^{-2}\xi_2,\,\xi_3) 
\begin{pmatrix} q^{-2} FK \\ q^{2}EK \\ -EF + \kappa^ 2K^2 - \kappa^ 2 K'^2 \end{pmatrix}  
\; = \; 0;
 \end{equation*}
 i.e., by the equation
 $ \xi_1 FK + \xi_2 EK + \xi_3(-EF + \kappa^ 2K^2 - \kappa^ 2
 K'^2)=0$.  The quadric $Q_p$ is singular at $p$ (and is therefore a cone)
 because the partial derivatives at the point
 $p=(\xi_1,\xi_2,\xi_3,0)$ are
 \begin{align*}
 \pd_E: \qquad & \; (\xi_2K -\xi_3F)\vert_p = 0,
 \\
 \pd_F: \qquad & \; (\xi_1K -\xi_3E)\vert_p = 0,
 \\
 \pd_K: \qquad & \; (\xi_1F +\xi_2E+2\xi_3\kappa^ 2K)\vert_p = 0, 
 \\
 \pd_{K'}: \qquad & \; 2\xi_3\kappa^ 2K'\vert_p = 0.
 \end{align*}
 It follows that every line contained in $Q_p$ passes through $p$, and
 hence all such lines correspond to line modules by
 \cite[Thm. 5.1.6]{LBSvdB}, as recalled above in
 \Cref{sssect.quadrics.Qp}.
 \end{proof}
 
 Let $C$ denote the conic $K'=\kappa^ 2K^2+EF=0$ and $C'$ the conic
 $K=\kappa^ 2K'^2+EF=0$.  The two isolated points in $\cP_D$, namely
 $(0,0,1,\pm 1)$, lie on $Q_p$ for every
 $p \in \{K'=\kappa^ 2K^2+EF=0\}$ so there is a pencil of lines (or
 two pencils) through each of these points that correspond to line
 modules.

 Since $Q_p$ is singular at $p$, \cite[Lem. 5.1.10]{LBSvdB} implies
 that $p=p^\vee$ in the notation defined on page 204 of {\it
   loc. cit.} Thus, according to \cite[Defn. 5.1.9]{LBSvdB}, $p$ is of
 the {third kind}.  Thus, we are in the last case of \cite[Table
 1]{LBSvdB}.

\section{Points, lines, and quadrics in $\Projnc(S)$}
\label{sec.S}
We now transfer the results in \S\ref{sec.D} from $D$ to $S$. 
Recall that the automorphism $\phi: S \to S$ defined in (\ref{eq:phi}) induces an equivalence of categories 
$\Phi:\Gr(S) \to \Gr(D)$. We first note how $\phi$ and $\Phi$ act on linear modules.

\subsection{Left ideals and linear modules over $S$ and $D$}
\label{ssect.modules.SD} 

If $W$ is a graded subspace of $S=D$, then $D_m*W_j = \phi^j(S_m)W_j = S_mW_j$ so, dropping the $*$, $DW=SW$, i.e., 
the left ideal of $D$ generated by $W$ is equal to the left ideal of $S$ generated by $W$. In particular, if $I$ is a graded left
ideal of $S$, then $I=SI=DI$ so $I$ is also a left ideal of $D$. Likewise, if $J$ is graded left
ideal of $D$, then $J=DJ=SJ$ so $J$ is also a left ideal of $S$. 

In summary, $D$ and $S$ have exactly the same left ideals.

Let $I$ be a graded left ideal of $S$. The equality $S/I=D/I$ is an equality in the category of graded vector spaces.
In fact, more is true: $\Phi(S/I)=D/I$ in the category $\Gr(D)$. To see this, observe, first, that the result of applying $\Phi$ to the exact sequence $0 \to I \to S \to S/I \to 0$ in $\Gr(S)$ is the exact sequence $0 \to I \to D \to \Phi(S/I) \to 0$ in $\Gr(D)$
where $I \to S$ and $I \to D$ are the inclusion maps, then use the fact that $\Phi(S/I) = S/I=D/I$ as graded vector spaces.

Now let $M$ be a $d$-linear $S$-module. Then $M \cong S/I$ for a unique graded left ideal $I$ in $S$. Hence $\Phi M \cong \Phi(S/I) = D/I$. In particular, $\Phi(M)$ is a $d$-linear $D$-module. Hence
$$
\{\hbox{left ideals $I$ in $S$} \; | \; S/I \hbox{ is $d$-linear}\}
\; = \;  \{\hbox{left ideals $I$ in $D$} \; | \; D/I \hbox{ is $d$-linear}\}.
$$
Similarly, if $S_0=k$, then
$$
\{\hbox{subspaces $W \subset S_1$} \; | \; S/SW \hbox{ is $d$-linear}\}
\; = \;  \{\hbox{subspaces $W \subset D_1$} \; | \; D/DW \hbox{ is $d$-linear}\}.
$$

\begin{lemma}
\label{lem.pts.lines} 
$\phantom{x}$
 \begin{enumerate}
  \item 
  $\cP_S=\cP_D$ and $\cL_S=\cL_D$.
  \item 
  If $\sigma_S:\cP_S \to \cP_S$ is a bijection such that 
  $(M_p)_{\ge 1}(1) \cong M_{\s_S^{-1}p}$  for all $p \in \cP_S$,
then there is a bijection $\sigma_D:\cP_D \to \cP_D$ such that $(M_p)_{\ge 1}(1)  \cong M_{\s_D^{-1}p}$ 
for all $p \in \cP_D$, namely $\s_D= \s_S\phi$
\end{enumerate}
\end{lemma}
\begin{proof}
(1)
This follows from the discussion prior to the lemma. 

(2) 
Let $p \in \cP_S$.
Suppose that $p=(\xi_0',\ldots,\xi_n')$  and $\s_S^{-1}(p)=(\xi_0,\ldots,\xi_n)$ with respect to an ordered basis 
$x_1,\ldots,x_n$ for $S_1$. There is a homogeneous basis $e_0,e_1,\ldots$, where $\deg(e_n)=n$, for the $S$-module $M_p$ such that $x_i e_0=\xi_i' e_1$ and $x_ie_1=\xi_i e_2$ for all $i$. 
Hence $(\xi_ix_j-\xi_jx_i)e_1=0$ for all $1 \le i,j \le n$. 

Since
\begin{align*}
\s_D^{-1}(p)^\perp & \; = \; \{x \in D_1 \; | \; x *e_1=0\}
\\
& \; = \; \{x \in D_1 \; | \; \phi(x)e_1=0\}
\\
& \; = \; \{\phi^{-1}(x) \in D_1 \; | \; xe_1=0\}
\\
& \; = \; \phi^{-1}\big(\{x \in D_1=S_1 \; | \; xe_1=0\}\big)
\\
& \; = \; \phi^{-1}\big(\s_S^{-1}(p)^\perp\big)
\\
& \; = \; \phi^{-1}(\s_S^{-1}(p))^\perp\, ,
\end{align*}
$\s_D^{-1} = \phi^{-1}\s_S^{-1}$ and $\s_D=\s_S\phi$. 
\end{proof}

\subsection{Points in $\Projnc(S)$, the point scheme of $S$, and point modules} 
We restate \Cref{lem.pts.lines}(1) explicitly in the following theorem.

\begin{theorem}
\label{thm.pts.S}
The point scheme $\cP_S$ is reduced. It is the union of 
\begin{enumerate}
  \item 
  the conics $EF+\kappa^ 2K^2=K'=0$ and $EF+\kappa^ 2K'^2=K=0$,
  \item 
  the line $K=K'=0$, and
  \item 
  the points $(0,0,1,\pm 1)$.
  \end{enumerate}
  Furthermore,
  \begin{enumerate}
  \item[(4)] 
 if $p=(\xi_1,\xi_2,\xi_3,0)$ is on the conic $EF+\kappa^ 2K^2=K'=0$, then $\s_S(p)= (q\xi_1,q^{-1}\xi_2,\xi_3,0)$;
 \item[(5)] 
  if $p=(\xi_1,\xi_2,0,\xi_4)$ is on the conic $EF+\kappa^ 2K'^2=K=0$,  then $\s_S(p)=(q^{-1}\xi_1,q\xi_2,0,\xi_4)$;
    \item
    [(6)]  
   if $p=(\xi_1,\xi_2,0,0)$ is on the line $K=K'=0$, then $\s_S(p)= p$;
  \item[(7)]
if $p=(0,0,1,\pm 1)$, then $\s_S(p)=p$.
\end{enumerate}
\end{theorem}
\begin{proof}
By \Cref{lem.pts.lines}, $\cP_S=\cP_D$. However, $\cP_D$ is reduced so $\cP_S$ is reduced.
By \Cref{thm.PD}, the irreducible components of $\cP_D$ are the varieties in parts (1), (2) and (3) of this theorem.
Hence the same is true of $\cP_S$. 

Now suppose there is an ordered basis $x_1,\ldots,x_n$ for $S_1$ and scalars $\l_1,\ldots,\l_n$ such that $\phi(x_i)=\l_ix_i$ for all $i$. 
Let $p=(\xi_1,\ldots,\xi_n) \in \PP(S_1^*)$ where the coordinates are written with respect to the ordered basis $x_1,\ldots,x_n$. Let $\xi$ be the point in $S_1^*$ with coordinates $(\xi_1,\ldots,\xi_n)$; i.e., $x_i(\xi)=\xi_i$ or, equivalently, 
$\xi(x_i)=\xi_i$. 
Since $\phi(\xi)(x_i)=\xi(\phi^{-1}x_i) = \xi(\l_i^{-1}x_i)$, $\phi(\xi)=(\l_1^{-1}\xi_1,\ldots,\l_n^{-1}\xi_n)$.
Hence 
$$\phi(p)=(\l_1^{-1}\xi_1,\ldots,\l_n^{-1}\xi_n).$$

Note that $E$, $F$, $K$, and $K'$ in $S$ are eigenvectors for $\phi$ with eigenvalues $q^{-1}$, $q$, $1$, and $1$, respectively. 
Thus if $p=(\xi_1,\xi_2,\xi_3,\xi_4) \in \PP(S_1^*)$ with respect to the ordered basis $E,F,K,K'$, then
$\phi(p)=(q\xi_1,q^{-1}\xi_2,\xi_3,\xi_4)$ and $\phi^{-1}(p) = (q^{-1}\xi_1,q\xi_2,\xi_3,\xi_4)$.

We now use the description of $\s_D$ in \Cref{thm.PD} to obtain:
\begin{enumerate}
\item  If $p=(\xi_1,\xi_2,\xi_3,0) \in \{EF+\kappa^ 2K^2=K'=0\}$, then $\s_D(p)= (q^2\xi_1,q^{-2}\xi_2,\xi_3,0)$ so 
 $\s_S(p)=\s_D\phi^{-1}(p)=(q\xi_1,q^{-1}\xi_2,\xi_3,0)$.
\item If $p=(\xi_1,\xi_2,0,\xi_4)\in \{EF+\kappa^ 2K'^2=K=0\}$,  then $\s_D(p)=p$ so 
  $\s_S(p)=\s_D\phi^{-1}(p)=(q^{-1}\xi_1,q \xi_2,\xi_3,0)$.
\item If $p=(\xi_1,\xi_2,0,0) \in \{K=K'=0\}$, then $\s_D(p)= (q\xi_1,q^{-1}\xi_2,0,0)$ so
 $\s_S(p)=\phi^{-1}(p)=(\xi_1,\xi_2,0,0)$.
\item If $p=(0,0,1,\pm 1)$, then $\s_D(p)=p$ so $\s_S(p)=\phi^{-1}(p)=p$. 
\end{enumerate}
The proof is complete.
\end{proof}

The algebra $D$ is less symmetric than $S$: the fact that $\s_D$ is
the identity on one of the conics but not on the other indicates a
certain asymmetry about $D$. The asymmetry is a result of the fact
that we favored $K'$ over $K$ when we formed the Zhang twist of $S$
which made $K'$, but not $K$, a central element. \Cref{thm.pts.S}
shows that the symmetry is restored when the results for $\cP_D$ are
transferred to $\cP_S$.

\subsection{Lines and quadrics in $\Projnc(S)$}
\Cref{prop.lines.D.1} classified the line modules for $D$, and
therefore the line modules for $S$.  \Cref{thm.lines.D} below gives a
new description of the line modules for $S$: it says that the line
modules correspond to the lines lying on a certain pencil of
quadrics. This is analogous to the description in \cite[Thm. 2]{LBS93}
of the line modules for the homogenized enveloping algebra of $\fsl_2$
and the description in \cite[Thm. 4.5]{LS93} of the line modules for
the 4-dimensional Sklyanin algebra $S(\alpha,\beta,\gamma)$.

The new description provides a unifying picture. The pencil of
quadrics becomes more degenerate as one passes from the
$S(\alpha,\beta,\gamma)$'s to the homogenizations of the various
$U_q(\fsl_2)$ and more degenerate still for $H(\fsl_2)$. The pencil
for $H(\fsl_2)$ contains a double plane $t^2=0$, that for $S$ contains
the pair of planes $\{KK'=0\}$, and that for $S(\alpha,\beta,\gamma)$
contains 4 cones and the other quadrics in the pencil are smooth.

The vertices of the cones in each pencil play a special role: for
$H(\fsl_2)$ there is only one cone and its vertex corresponds to the
trivial representation of $U(\fsl_2)$; for $S$ there are two cones and
their vertices correspond to the two 1-dimensional representations of
$U_q(\fsl_2)$; for $S(\alpha,\beta,\gamma)$ the vertices of the four
cones ``correspond'' to the four special 1-dimensional
representations.

\subsubsection{}
The line through two points $(\xi_1,\xi_2,\xi_3,0)$ and
$(\eta_1,\eta_2,0,\eta_4)$ in $\PP(S_1^*)$ is given by the equations
\begin{equation}
\label{line.mod.S}
\xi_3\eta_4 E \, - \, \xi_1\eta_4 K \,- \, \xi_3\eta_1 K' \; = \; \xi_3\eta_4 F \, - \, \xi_2\eta_4 K \,- \, \xi_3\eta_2 K' \; = \; 0.
\end{equation}
 
\subsubsection{The pencil of quadrics $Q(\l)\subseteq \PP(S_1^*)$}
For each $\l \in \PP^1$, let $Q(\l) \subseteq \PP(D_1^*)$ be the quadric where 
$$
g_\l \; := \; \kappa^ {-2}EF + K^2 -(\l+\l^{-1})KK'+K'^2 
$$
vanishes. The points on the conics
\begin{align*}
C': \qquad   &  \kappa^ 2K'^2 +EF \; =\; K\; = \; 0 \qquad \hbox{and}
\\
C: \qquad   &  \kappa^ 2K^2 +EF \; =\; K'\; = \; 0 
\end{align*}
correspond to point modules for $S$. If $\l \ne 0, \infty$, then 
$$
C' \; = \; Q(\l) \cap \{K=0\}  \qquad \text{and} \qquad C \; = \; Q(\l) \cap \{K'=0\}.
$$

\begin{proposition} 
\label{prop.pencil}
$\phantom{xx}$
\begin{enumerate}
  \item 
The base locus of the pencil $Q(\l)$ is $C \cup C'$. 
  \item 
The $Q(\l)$'s are  the only quadrics that contain $C \cup C'$. 
  \item 
The singular quadrics in the pencil  are the cones $Q(\pm 1)$ with vertices at $(0,0,1,\pm 1)$ respectively, 
and $Q(0)=Q(\infty)=\{KK'=0\}$.
  \item{}
The lines on $Q(1)$ are $\kappa^ {-1}E -s(K-K')  =  s\kappa^ {-1}F+(K-K')=0$, $s \in \PP^1$, and the lines on $Q(-1)$ are $\kappa^ {-1}E -s(K+K')  =  s\kappa^ {-1}F+(K+K')=0$, $s \in \PP^1$.
 \item{}
  Suppose $\l \notin \{0,\pm 1,\infty\}$. The two rulings on $Q(\l)$ are
\begin{align}
\label{ruling.1}
\kappa^ {-1}E -s(K-\l K') \; \,  &= \,  \; s\kappa^ {-1}F+(K-\l^{-1}K') \; = \; 0,  \; \quad s \in \PP^1, \qquad \hbox{and}
\\
\label{ruling.2}
\kappa^ {-1}E -s(K-\l^{-1} K') \; &= \; s\kappa^ {-1}F+(K-\l K') \; = \; 0,  \; \quad s \in \PP^1.
\end{align}
\end{enumerate}
\end{proposition}
\begin{proof}
(1)
The base locus is, by definition, the intersection of all $Q(\mu)$ so is given by the equations $KK'=\kappa^ {-2}EF + K^2 +K'^2=0$ so is
$\{K'=\kappa^ {-2}EF + K^2 =0\} \cup \{K=\kappa^ {-2}EF + K'^2 =0\}$. 

The proofs of (2) and (3) are straightforward.  To prove (3) observe that the determinant of the symmetric matrix representing the  
$g_\l$ has zeroes at $\l=\pm 1$ and a zero at $\l=0,\infty$.

(4) and (5).
Let $\ell$ be the line defined by (\ref{ruling.1}). 

Suppose $s \notin\{0,\infty\}$. Then $\kappa^ {-1}E=s(K-\l K')$ and
$s\kappa^ {-1}F=-(K-\l^{-1}K')$ on $\ell$ so
$s\kappa^ {-2}EF=-s(K-\l K')(K-\l^{-1}K')$ on $\ell$.  Cancelling $s$,
this says that $\kappa^ {-2}EF+(K-\l K')(K-\l^{-1}K')$ vanishes on
$\ell$.  Since the equation for $Q(\l)$ can be written as
$\kappa^ {-2}EF +( K-\l K')(K-\l^{-1}K') =0$, $\ell \subseteq Q(\l)$.
If $s=0$, then $\ell$ is the line $E=K-\l^{-1}K'=0$ which is on
$Q(\l)$.  If $s=\infty$, then $\ell$ is the line $K+\l^{-1}K'=F=0$
which is on $Q(\l)$.

The other case (\ref{ruling.2}) is similar. 
\end{proof}

\subsubsection{}
There are exactly four singular quadrics in a generic pencil of quadrics in $\PP^3$. The point modules for the 4-dimensional Sklyanin algebras $S(\alpha,\beta,\gamma)$ are parametrized by a quartic elliptic curve $E \subseteq \PP^3$ and 4 isolated points that are the vertices of the  singular quadrics that contain $E$. The point modules corresponding to those isolated points correspond to the four 1-parameter families of 1-dimensional representations of a 4-dimensional Sklyanin algebra. 

\subsubsection{}
The vertices of the cones $Q(\pm 1)$ are the points $(0,0,1,\pm1)$. These are the isolated points in the point scheme $\cP_S$ (see \Cref{thm.pts.S}). Later, we will see that  the points $(0,0,1,\pm1)$ correspond to the two 1-dimensional $U_q(\fsl_2)$-modules. 
More precisely, if $p$ is one of those points, then $M_p[(KK')^{-1}]_0$ is a 1-dimensional $U_q(\fsl_2)$-module.

\subsubsection{}
\label{sssect.lines.Q1}
The lines on $Q(1)$ meet $C'$ and $C$ at points of the form $(\xi_1,\xi_2,\xi_3,0)$ and $(\xi_1,\xi_2,0,-\xi_3)$ respectively. 
The lines on $Q(-1)$ meet $C'$ and $C$ at points of the form $(\xi_1,\xi_2,\xi_3,0)$ and $(\xi_1,\xi_2,0,\xi_3)$ respectively.

\begin{theorem}
\label{thm.pencil.lines}
Let $\ell$ be a line in $\PP(S_1^*)$. Then $S/S\ell^\perp$ is a line module if and only if $\ell \subseteq Q(\l)$ for some $\l \in \PP^1$.
\end{theorem}
\begin{proof}
($\Rightarrow$)
Suppose $S/S\ell^\perp$ is a line module. By \S\ref{ssect.modules.SD},  $D/D\ell^\perp$ is a line module for $D$. 

The result is true if $\ell \subseteq \{KK'=0\}=  Q(\infty)$ so, from now on, suppose $\ell \not\subseteq \{KK'=0\}$.

Let $p=(\xi_1,\xi_2,\xi_3,0)$ be the point where $\ell$ meets $\{K'=0\}$.
By the discussion at the beginning of \S\ref{ssect.D.lines}, $A/Ap^\perp$ is a point module for $A=D/(K')$ so $p \in C\cup \{K=K'=0\}$. 

Suppose $p=(1,0,0,0)$. By \Cref{prop.lines.D.1}(1), $\ell$ is on the plane $\{F=0\}$. Since $p \in \ell$ it follows that $\ell=\{F= K- \l K'=0\}$ for some $\l \in \PP^1$. Since  $\ell \not\subseteq \{KK'=0\}$, $\l \ne 0,\infty$. 
Thus $\ell$ is the line in (\ref{ruling.1}) corresponding to the point  $s=\infty \in \PP^1$. 
Hence $\ell \subseteq Q(\l)$. 

If $p=(0,1,0,0)$, a similar argument shows that $\ell$ lies on some $Q(\l)$.

Now suppose that $p \notin\{(1,0,0,0),(0,1,0,0)\}$.
Since  $\ell \not\subseteq \{KK'=0\}$, it follows  from \Cref{prop.lines.D.1}(3) that $\xi_3 \ne 0$.  
Hence by \Cref{prop.lines.D.1}(4), $\ell$ lies
on the quadric 
$$
\xi_1FK+\xi_2EK+\xi_3(-EF+\kappa^ 2K - \kappa^ 2 K'^2)=0.
$$ 
The conic $C'=\{K=EF+\kappa^ 2K'^2=0\}$ also lies on this quadric so $C' \cap \ell \ne \varnothing$. 
by a result analogous to \Cref{prop.lines.D.1}(3), $\eta_4 \ne 0$. 
Let $(\eta_1,\eta_2,0,\eta_4) \in C' \cap \ell$. 
Then $\ell$ is given by the equations in (\ref{line.mod.S})
so lies on the surface cut out by the equation
\begin{align*}
\xi_3^2\eta_4^2 EF &\; = \; ( \xi_1\eta_4 K +\xi_3 \eta_1 K') (\xi_2\eta_4 K +\xi_3 \eta_2 K' )
\\ 
 &\; = \;  \xi_1\xi_2\eta_4^2 K^2 +  \xi_3\eta_4(  \xi_1\eta_2 +  \xi_2\eta_1)KK' +
\xi_3^2\eta_1\eta_2 K^{'2}
\\
 &\; = \;  -\kappa^ {2}\xi_3^2\eta_4^2 K^2 + \xi_3 \eta_4(  \xi_1\eta_2 +  \xi_2\eta_1)KK' 
 -\kappa^ {2}\xi_3^2\eta_4^2 K^{'2}
\end{align*}
which can be rewritten as $\xi_3^2\eta_4^2(\kappa^ {-2} EF+K^2+K^{'2}) - \xi_3 \eta_4(  \xi_1\eta_2 +  \xi_2\eta_1)KK' =0$.
Thus, $\ell$ lies on some $Q(\l)$.  

($\Leftarrow$) 
Let $\ell$ be a line on $Q(\l)$. If $\ell \subseteq \{K'=0\}$, then $D/D\ell^\perp=A/A\ell^\perp$ {\it is} a line module. From now on
suppose that $\ell \not\subseteq \{K'=0\}$.

To show that $D/D\ell^\perp$, and hence $S/S\ell^\perp$, is a line module we must show
there is a point, $p$ say, in $\ell \cap C$ such that $\ell \subseteq Q_p$ where $Q_p$ is the quadric in 
\S\ref{sssect.quadrics.Qp}.

Suppose $\ell \subseteq Q(1)$. By \Cref{prop.pencil}(4), $\ell$ is given by
$$
\kappa^ {-1}E-s(K-K')=s\kappa^ {-1}F+(K-K')=0
$$
for some $s \in \PP^1$. Since the point $p=(-s^2,1,-s\kappa^ {-1},0)$ belongs to $\ell \cap C$, $S/S\ell^\perp$ is a line module if and only if 
$\ell \subseteq Q_{p}$. Since $Q_p$ is given by the equation
$$
-s^2FK+EK-s\kappa^ {-1}(-EF+\kappa^ 2K^2-\kappa^ 2K'^2)=0,
$$
the point $(-s^2,1,0, s\kappa^ {-1})$ is in $\ell \cap Q_{p}$.  Thus, $\ell$ passes through the vertex of the cone $Q_p$ and through a second point on $Q_p$, whence $\ell \subseteq Q_p$. Therefore $S/S\ell^\perp$ is a line module.

The case $\ell \subseteq Q(-1)$ is similar.

Suppose $\l \notin \{0,\pm 1,\infty\}$. Since $\ell \subseteq Q(\l)$ we suppose, without loss of generality, 
that $\ell$ belongs to the ruling \Cref{ruling.2} on $Q(\l)$. 
Thus $\ell=\{\kappa^ {-1}E -s(K-\l K')  = s\kappa^ {-1}F+(K-\l^{-1} K') =0\}$ for some $s \in \PP^1$.
The point $p=(-\kappa s^2,\kappa,-s,0)$, which is the vertex of the cone $Q_p$, is in $\ell \cap C$. Thus $\ell$ passes through the vertex of $Q_p$
and the point $(-\kappa s^2\nu^2,\kappa,0,s\nu)$ which is also on $Q_p$,   $\ell \subseteq Q_p$.  Hence $S/S\ell^\perp$ is a line module.
\end{proof}

 \begin{theorem}
 \label{thm.lines.D}
 Let $\ell$ be a line in $\PP(S_1^*)$. Then $S/S\ell^\perp$ is a line module if and only if  
  $\ell$ meets $C \cup C'$ with multiplicity $2$; i.e., if and only if $\ell$ is a secant line to $C \cup C'$.   
 \end{theorem} 
 \begin{proof}
 ($\Rightarrow$)
 Since $S/S\ell^\perp$ is a line module for $S$, $D/D\ell^\perp$ is a line module for $D$.  
 Let $\l$ be such that  $\ell \subseteq Q(\l)$. 
 
 Suppose $Q(\l)$ is smooth. The Picard group of $Q(\l)$ is isomorphic to 
$\ZZ \times \ZZ$ and equal to $\ZZ[L] \oplus \ZZ[L']$ where $[L]$ is the class of the line in (\ref{ruling.1}) corresponding to $s=0$
and $[L']$ is the class of the line in (\ref{ruling.2}) corresponding to $s=0$. 
Since $L=\{E=K-\l^{-1}K'=0\}$, the scheme-theoretic intersection $L\cap(C\cup C')$ is the zero locus of the ideal 
\begin{align*}
(E,K-\l^{-1} K') + (KK',g_\l)  
&\; = \;  (E,K-\l^{-1} K',KK', \kappa^ {-2}EF+(K-\l K')(K-\l^{-1}K')) 
\\
&\; = \;   (E,K-\l^{-1} K',KK').
\end{align*}
Hence $L\cap(C\cup C')$ is a finite scheme of length 2. Therefore $[L]\cdot [C \cup C']=2$.
A similar calculation shows that $[L']\cdot [C \cup C']=2$.
Hence $[C \cup C']=2[L] +2 [L']$. It follows that  $[\ell]\cdot[C \cup C']=2$.

Suppose $\ell$  lies on the cone $Q(1)$. 
Then $\ell$ is the line 
$
\kappa^ {-1}E -s(K+K')  =  s\kappa^ {-1}F+(K+K')=0
$ 
for some $s \in \PP^1$. Therefore the scheme-theoretic intersection $\ell \cap (C \cup C')$ is the zero locus of the ideal
\begin{equation}
\label{an.ideal}
(\kappa^ {-1}E -s(K-K'),  s\kappa^ {-1}F+(K-K'))  \; + \; (KK',g_1).
\end{equation}
Since $\ell \subset Q(1)$, $g_1$ belongs to the ideal vanishing on $\ell$. The 
ideal in (\ref{an.ideal}) is therefore equal to $(\kappa^ {-1}E -s(K-K'),  s\kappa^ {-1}F+(K-K'),KK')$.
Thus $\ell \cap(C\cup C')$ is a finite scheme of length 2.

If $\ell \subseteq Q(-1)$, a similar argument shows that $\ell \cap(C\cup C')$ is a finite scheme of length 2. 

 Suppose $\ell \subseteq  Q(\infty)=\{KK'=0\}$.  Without loss of generality we can, and do, assume that $\ell \subseteq \{K'=0\}$.
 By B\'ezout's Theorem, $\ell$ meets $C$ with multiplicity two. Thus, if $\ell \cap C'=\varnothing$, then $\ell$ meets
 $C \cup C'$ with multiplicity  two. Now suppose that $\ell$ meets $C$ with multiplicity two and $C'$ with
 multiplicity $\ge 1$. If $\ell$ meets $C$ at two distinct points, then $\ell$ is transversal to some $Q(\l')$ so meets $Q(\l')$,
 and hence $C \cup C'$, with multiplicity two. It remains to deal with the case where $\ell$ is tangent to $C$ and meets $C'$. We now assume that is the case. Since $C \cap C'=\{(1,0,0,0),(0,1,0,0)$ it follows that $\ell$ is tangent to $C$ at $(1,0,0,0)$ or at $(0,1,0,0)$. Since the two cases are similar we assume that $\ell$ is tangent to $C$ at $(1,0,0,0)$. It follows that $\ell=\{K'=F=0\}$.  
The scheme-theoretic intersection $\ell \cap (C \cup C')$ is the zero locus of the ideal
$$
(K,F)+(KK',g_1) \; = \;  (K,F,KK',\kappa^ {-2}EF+K^2+K'^2) \; = \; (K,F,K'^2).
$$
Hence $\ell \cap(C\cup C')$ is a finite scheme of length 2.

 ($\Leftarrow$)
 Suppose  $\ell$ is a line that meets $C \cup C'$ with multiplicity $2$. 
 
If $\ell$ lies on the plane $\{K'=0\}$, then $K' \in \ell^\perp$ so the ideal $(K')$ of $D$ is contained in $D\ell^\perp$ and 
$D/D\ell^\perp$ is a module over $A=D/(K')$. However, the dual of the map $D_1 \to A_1$ embeds $\PP(A_1^*)$ in $\PP(D_1^*)$ 
and the image of this embedding is $\{K'=0\}$. In short, $\ell$ is a line in $\PP(A_1^*)$. 
This implies that $A/A\ell^\perp$ is a line module for $A$ and hence a line module for $D$. But $A/A\ell^\perp=D/D\ell^\perp$ so 
$D/D\ell^\perp$ is a line module for $D$. Therefore $S/S\ell^\perp$ is a line module for $S$. 
A similar argument shows that if  $\ell$ lies on the plane $\{K=0\}$, then $S/S\ell^\perp$ is a line module.

For the remainder of the proof we assume that $\ell \not\subseteq \{KK'=0\}$.
 
A line that meets $C$ with multiplicity two lies in the plane $K'=0$, so $\ell$ meets $C$ and $C'$ with multiplicity one. Let $p=(\xi_1,\xi_2,\xi_3,0)$ be the point where $\ell$ meets $C$ and $p=(\eta_1,\eta_2,0,\eta_4)$ be the point where $\ell$ meets $C'$.
Since $p$ is the vertex of $Q_p$ and $C' \subseteq Q_p$, $\ell$ passes through the vertex of $Q_p$ and another point on $Q_p$. 
Hence $\ell \subseteq Q_p$. 
 \end{proof}

 \subsubsection{}
 \Cref{thm.pencil.lines,thm.lines.D} are analogous to results for the  4-dimensional Sklyanin algebras:
 the line modules correspond to the lines in 
 $\PP^3$ that lie on the quadrics that contain the quartic elliptic curve $E$,
 and those are exactly the lines in $\PP^3$ that meet $E$ with multiplicity two, i.e., the secant lines to $E$.
Similar results hold for the homogenization of $\fsl_2$ \cite{LBS93}.

\subsubsection{Notation for line modules}
By \Cref{thm.lines.D},  the lines that correspond to line modules for $S$ are the secant lines to $C \cup C'$. If $(p)+(p')$ is a degree-two divisor
on $C \cup C'$ we write $M_{p,p'}$ for the line module $M_\ell=S/S\ell^\perp$ 
where $\ell$ is the unique line that meets $C \cup C'$ at $(p)+(p')$. Thus, up to isomorphism, the line modules for $S$ are
$$
\{M_{p,p'} \; | \; p,p' \in C \cup C'\}.
$$

\subsection{Incidence relations between lines and points in $\Projnc(S)$}  
Let $(p)+(p')$ be a degree-two divisor on $C \cup C'$. 
There is a surjective map $M_{p,p'} \twoheadrightarrow M_{p}$ in $\Gr(S)$ and, by 
 \cite[Lemma 5.3]{LS93}, the kernel of that homomorphism is isomorphic to $M_{\ell'}(-1)$ for some $\ell'$. 
 Our next goal is to determine $\ell'$. We do that in  \Cref{lem.incidence} below.
 
 First we need the rather nice observation in the next lemma. 
 
We call a degree-three divisor on a plane cubic curve {\sf linear} if it is the scheme-theoretic intersection of that curve and a line.

\begin{lemma}
\label{lem.autom.conic.line}
Let $C$ be a non-degenerate conic in $\PP^2$, $\s$ an automorphism of $C$ that fixes two points, and $L$ the line through those two points. 
Let $p,p' \in C$ and $p'' \in L$. The divisor $(p) +(p')+(p'') \in {\rm Div}(C \cup L)$ is linear if and only if $(\s p)+(\s^{-1} p')+(p'')$ is.  
\end{lemma}
\begin{proof}
By symmetry, it suffices to show that if $(p) +(p')+(p'')$ is linear so is $(\s p)+(\s^{-1} p')+(p'')$. That is what we will prove. So, assume 
$(p) +(p')+(p'')$ is linear.

If $\tau$ is an automorphism of $\PP^1$ that fixes two points, there are non-zero scalars $\l$ and $\mu$ and a choice of coordinates such that 
$\tau(s,t)=(\l s, \mu t)$ for all $(s,t) \in \PP^1$. We assume, without loss of generality, that $(C,\s)$ is the image of $(\PP^1,\tau)$ 
under the 2-Veronese embedding. Thus, we can assume that $C$ is the curve $xy-z^2=0$ and $\s(\a,\b,\c)=(\l^2 \a, \mu^2 \b, \l\mu \c)$. 
The line $L$ is the line  through $(1,0,0)$ and $(0,1,0)$. 

Let $p=(\a,\b,\c)$,  $p'=(\a',\b',\c')$, and $p''=(a,b,0)$. By hypothesis, these three points are collinear. Therefore
$$
\det \begin{pmatrix} 
a & b & 0 \\
\a & \b & \c \\
\a' & \b' & \c'
\end{pmatrix} \; = \; 0.
$$
I.e., $a(\b\c'-\b'\c)-b(\a\c'-\a'\c) =0$.

To show that $(\s p)+(\s^{-1} p')+(p'')$ is linear we must show that the points $\s p=(\l^2 \a,\mu^2 \b,\l\mu \c)$,  $\s^{-1} p'=(\l^{-2}\a',\mu^{-2}\b',\l^{-1}\mu^{-1}\c')$, and $p''$,  are collinear. This is the case if and only if
$$
\det \begin{pmatrix} 
a & b & 0 \\
\l^2\a & \mu^2\b & \l\mu\c \\
\l^{-2}\a' & \mu^{-2}\b' & \l^{-1}\mu^{-1}\c'
\end{pmatrix} \; = \; 0.
$$
This determinant is $a(\l^{-1}\mu \b\c'-\l\mu^{-1}\b'\c)-b(\l\mu^{-1}\a\c'-\l^{-1}\mu\a'\c)$. 
It is zero if and only if  
$$
(\a\c'-\a'\c)(\l^{-1}\mu \b\c'-\l\mu^{-1}\b'\c)   \; - \;   (\b\c'-\b'\c)(\l\mu^{-1}\a\c'-\l^{-1}\mu\a'\c)  \;=\; 0.
$$
This expression is  equal to  
$$
\l^{-1}\mu (\a\b\c'^2-\a'\b'\c^2)   \; + \;  \l\mu^{-1} (\a'\b'\c^2 - \a\b\c^2).
$$
 But $\a\b-\c^2 = \a'\b'-\c'^2=0$ so $\a\b\c'^2 - \a'\b'\c^2=0$.
Thus, the determinant is zero and we conclude that $\s^{-1}p'$, $\s p$, and $p''$, are collinear. 
\end{proof}

\begin{remark}
For an alternative approach to \Cref{lem.autom.conic.line}, note first that the statement can be recast as the claim that if $\eta$ is the involution of $C$ obtained by ``reflection across $p''$'' meaning that
\begin{equation*}
 \eta(p) = \text{the second intersection of the line }pp''\text{ with }C,
\end{equation*}
then $\pi=\eta\circ\sigma$ is an involution. 

In turn, the involutivity of $\pi$ follows from the fact that it
interchanges the points $p$ and $p'$, and any automorphism of $\bP^1$
that interchanges two points is, after a coordinate change identifying
said points with $0,\infty$, of the form
$\displaystyle z\mapsto \frac tz$ for some constant $t$.
\end{remark}

The next result is analogous to \cite[Thm. 5.5]{LS93} which shows for
the 4-dimensional Sklyanin algebras that if $(p)+(p')$ is a degree-two
divisor on the quartic elliptic curve $E$, then there is an exact
sequence
\begin{equation*}
0 \to M_{p+\tau,p'-\tau}(-1) \to M_{p,p'} \to M_p \to 0  
\end{equation*}
where $\tau$ is the point on $E$ such that $\s p=p+\tau$ for all
$p \in E$.

\begin{proposition}\label{pr.res_CC'}
\label{lem.incidence}
If $(p)+(p')$ is a degree-two divisor on $C \cup C'$, there is an
exact sequence
\begin{align*}
  0 \longrightarrow M_{\sigma p,\sigma^{-1}p'}(-1) \longrightarrow M_{p,p'} \longrightarrow M_p \longrightarrow 0.
\end{align*}
\end{proposition}
\begin{proof}
  Let $\ell$ be the unique line in $\PP^3=\PP(S_1^*)$ such that
  $\ell \cap C=(p)+(p')$; i.e., $M_{p,p'}=M_\ell$.  By \cite[Lemma
  5.3]{LS93}, there is an exact sequence
  $ 0 \to M_{\ell'}(-1) \to M_\ell \to M_p \to 0$ for some line module
  $M_{\ell'}$.  We complete the proof by showing that
  $\ell' \cap C=(\s p)+(\s^{-1}p')$; i.e.,
  $M_{\ell'} = M_{\sigma p,\sigma^{-1}p'}$.  There are several cases
  depending on the location of $p$ and $p'$.

{\bf Case 0.}  Suppose $\ell=L$. Then $KM_\ell=K'M_\ell=0$ so
$M_\ell$, and consequently $M_{\ell'}$, is a module over $S/(K,K')$.
Since $S/(K,K')$ is a commutative polynomial ring on two
indeterminates it has a unique line module up to isomorphism, itself.
In particular, $\ell'=\ell$.  Hence there is an exact sequence
$0 \to M_{p,p'}(-1) \longrightarrow M_{p,p'} \longrightarrow M_p
\longrightarrow 0$.  But $\s$ is the identity on $L$ by
\Cref{thm.pts.S}, so $M_{p,p'}=M_{\sigma p,\sigma^{-1}p'}$. Thus, the
previous exact sequence is exactly the sequence in the statement of
this proposition.

{\bf Case 1.}  Suppose $p,p' \in C$. Then $\ell$ meets $C$ with
multiplicity two and therefore the plane $\{K'=0\}$ with multiplicity
$\ge 2$.  Hence $\ell \subseteq \{K'=0\}$. It follows that $M_\ell$
and $M_{\ell'}$ are modules over $S/(K')$. Given Case 0 treated above,
for the remainder of Case 1 we can, and do, assume that $\ell \ne L$.


Since $\ell \ne L$, $\ell \cap (C + L)=(p)+(p')+(p'')$ where $p''$ is
the point where $\ell$ and $L$ meet.  Since $S/(K')$ is a
3-dimensional Artin-Schelter regular algebra, \cite[Prop. 6.24]{ATV2}
tells us that $\ell'$ is the unique line such that $\ell' \cap (C+ L)$
contains the divisor $(\s^{-1} p')+(\s^{-1} p'')$.\footnote{Since
  \cite[Prop. 6.24]{ATV2} is for right modules and we are working with
  left modules we replaced $\s$ by $\s^{-1}$ in the conclusion of that
  result.}  By \Cref{thm.pts.S}, $\s p''=p''$ so $\ell'$ is the unique
line in $\{K'=0\}$ such that $\ell' \cap (C+ L)$ contains
$(\s^{-1} p')+(p'')$.  By \Cref{lem.autom.conic.line}, $\s p$,
$\s^{-1}p'$, and $p''$, are collinear.  Therefore
$\ell' \cap (C+L)=(\s p)+(\s^{-1}p')+(p'')$.

{\bf Case 2.}  If $p,p' \in C'$, the ``same'' argument as in Case 1
proves the proposition.

{\bf Case 3.}  Suppose $p \in C-C'$ and $p' \in C'-C$. Let
$p = (\xi_1,\xi_2,\xi_3,0)$ and $p'=(\eta_1,\eta_2,0,\eta_4)$.  Since
$p \notin C'$, $\xi_3 \ne 0$.  Since $p' \notin C$, $\eta_4 \ne 0$.

By \Cref{line.mod.S}, $\ell$ is given by the equations
\begin{align*}
X &:= \xi_3 \eta_4 E - \xi_1 \eta_4 K - \xi_3 \eta_1 K' = 0,\\
Y &:= \xi_3\eta_4 F - \xi_2 \eta_4 K - \xi_3 \eta_2 K' = 0.
\end{align*}
The corresponding linear modules are
\begin{align*}
M_\ell &= M_{p,p'} = \frac{S}{SX + SY},\\
M_p &= \frac{S}{SK' + SX + SY}, \\
M_{p'} &= \frac{S}{SK + SX + SY}.
\end{align*}

By \Cref{line.mod.S}, the line through 
$\sigma p= (q \xi_1,q^{-1} \xi_2,\xi_3,0)$ and $ \sigma^{-1}p' = (q \eta_1,q^{-1} \eta_2,0,\eta_4)$ is $\{X' = Y'=0\}$ where
\begin{align*}
X' &:= \xi_3 \eta_4 E - q\;\xi_1 \eta_4 K - q\;\xi_3 \eta_1 K',\\
Y' &:= \xi_3\eta_4 F - q^{-1}\xi_2 \eta_4 K - q^{-1}\xi_3 \eta_2 K',
\end{align*}
The corresponding line module is $M_{\sigma p,\sigma^{-1}p'} =  S/SX' + SY'$.

The image of $K'$ in $M_\ell$ generates the kernel of $M_\ell \to M_p$. Since $X'K' = q K' X$ and $Y'K' = q^{-1} K'Y$, $X'$ and $Y'$ annihilate the 
image of $K'$ in $M_\ell$. It follows that there is  a map from $M_{\sigma p,\sigma^{-1}p'}(-1)$ onto the kernel of $M_\ell \to M_p$.  
Thus, the kernel of $M_\ell \to M_p$ is isomorphic to a quotient of $M_{\sigma p,\sigma^{-1}p'}$. But every non-zero submodule of a line modules has GK-dimension 
2, and every proper quotient of a line module has GK-dimension  1, so every non-zero homomorphism map $M_{\sigma p,\sigma^{-1}p'}(-1) \to M_\ell$ is injective.
This shows that the kernel of $M_\ell \to M_p$ is isomorphic to $M_{\sigma p,\sigma^{-1}p'}(-1)$.

{\bf Case 4.}
If $p' \in C-C'$ and $p \in C'-C$,  the ``same'' argument as in Case 3 proves the proposition.
\end{proof}

We continue to write $L$ for the line $\{K=K'=0\}$.

\begin{proposition}
\label{prop.pt.line.incidence}
Let $\ell$ be a line in $\Projnc(S)$\footnote{This means that $M_\ell:=S/S\ell^\perp$ is a line module.} and suppose $p'' \in \ell \cap L$. 
\begin{enumerate}
  \item 
There are points $p,p' \in C \cup C'$  such that the scheme-theoretic intersection $\ell \cap (C \cup L)$ contains the divisor $(p)+(p')+(p'')$.
  \item 
There is an exact sequence 
 $$
 0 \longrightarrow M_{\s^{-1}p,\s^{-1}p'} (-1) \longrightarrow M_{p,p'} \to M_{p''} \to 0.
 $$
\end{enumerate}
 \end{proposition} 
 \begin{proof}
 Since $M_\ell$ is a line module, $\ell$ meets $C \cup C'$ with multiplicity two. It therefore meets either $C\cup L$ or $C'\cup L$ with multiplicity $\ge 2$. 
 We can, and do, assume without loss of generality that $\ell$ meets $C \cup L$ with multiplicity $\ge 2$. Hence $\ell$ meets the plane $\{K'=0\}$ with 
 multiplicity $\ge 2$. Therefore $\ell \subseteq \{K'=0\}$. By B\'ezout's theorem, $\ell$ is either equal to $L$ or meets $C \cup L$ with multiplicity $3$. 
 
 Suppose $\ell = L$. 
 Then $M_\ell$ is a module over the commutative polynomial ring $S/(K,K')$ and there is an exact sequence $0 \to M_\ell(-1) \to M_\ell \to M_{p''} \to 0$.
 Let $p$ and $p'$ be the points where $L$ meets $C\cup C'$. Then $M_\ell = M_L=M_{p,p'}$ and, since $\s$ is the identity on $L$, $M_\ell = M_{\s p,\s^{-1}p'}$.
 Thus, (1) and (2) hold when $\ell=L$.
 
  Suppose $\ell \ne L$. Let $p$ and $p'$ be the points in $\ell \cap C$; i.e., $\ell \cap C=(p)+(p')$. By \cite[Prop. 6.24]{ATV2}, there is an exact sequence
   $0 \to M_{\ell'}(-1) \to M_\ell = M_{p,p'} \to M_{p''} \to 0$ where $\ell'$ is the unique line whose scheme-theoretic intersection with $C \cup L$ is $\ge 
   (\s^{-1} p)+(\s^{-1} p')$.  Hence $M_{\ell'}=M_{\s^{-1}p,\s^{-1}p'}$.  
 \end{proof}

\section{Relation to $U_q(\fsl_2)$-modules}
\label{sec.Uq}
In this section, we relate our results about fat point and line modules for $S$ to classical results about the finite dimensional irreducible representations and Verma modules of $U_q(\fsl_2)$. Briefly, fat points in $\Projnc(S)$ correspond to finite dimensional irreducible $U_q(\fsl_2)$-modules and lines in 
 $\Projnc(S)$ correspond to Verma modules.  

\subsection{Facts about $U_q(\fsl_2)$}
First, we recall a few facts about $U_q(\fsl_2)$ that can be found in \cite[Chap. 2]{J96}.  

\subsubsection{Verma modules}
For each $\lambda\in \CC$, we call
\begin{equation*}
  M(\lambda) := \frac{U_q(\fsl_2)}{U_q(\fsl_2)e+U_q(\fsl_2)(k-\l)}
\end{equation*}
a {\sf Verma module} for $U_q(\fsl_2)$, and $\lambda$ its {\sf highest weight}.

\subsubsection{Casimir element}
The {\sf Casimir element}
\begin{equation}
\label{eq.Casimir}
C\; := \; ef+\frac{q^{-1}k+qk^{-1}}{(q-q^{-1})^2} \; = \;  fe+\frac{qk+q^{-1}k^{-1}}{(q-q^{-1})^2}
\end{equation}
is in the center of $U_q(\fsl_2)$ and acts on $M(\l)$ as multiplication by   
$$\frac{q\l+q^{-1}\l^{-1}}{(q-q^{-1})^2}.$$

\subsubsection{Finite dimensional simples}\label{ssse.fin}
For each $n \geq 1$, there are exactly two simple $U_q(\fsl_2)$-modules of dimension $n+1$. They can be labelled $L(n,+)$ and $L(n,-)$ in such a way that there are exact sequences
\begin{equation}
\label{eq.BGG}
0 \to M\left(\pm q^{-n-2}\right) \to M\left(\pm q^n\right) \to L(n,\pm) \to 0.
\end{equation}

The module $L(n,\pm)$ has basis $m_0,\ldots,m_n$ with action
\begin{equation}
\label{eq.Uq.simple.action}
k m_i= \pm q^{n-2i} m_i, 
\qquad f m_i= \begin{cases} m_{i+1} & \text{if $i<n$} \\ 0  & \text{if $i=n$,} \end{cases}
\qquad e m_i= \begin{cases} \pm [i][n+1-i] m_{i-1} & \text{if $i>0$} \\ 0  & \text{if $i=0$,} \end{cases}
\end{equation}
where we have made use of the {\sf quantum integers}
$$
[m] \; := \; \frac{ q^m-q^{-m}}{q-q^{-1}}\, .
$$

\subsection{Lines in $\Projnc(S)$ $\longleftrightarrow$ Verma modules for $U_q(\fsl_2)$}
\label{ssect.lines.Vermas}
First, we show that Verma modules are ``affine pieces'' of  line modules.

\begin{proposition}\label{pr.vrm-ln}
Let $\l \in \CC \cup \{\infty\} = \PP^1$ and let $\ell$ be the line $E=K-\l K'=0$. 
\begin{enumerate}
  \item 
 $\ell$ lies on the quadric $Q(\l)$. 
  \item 
$S/S\ell^\perp$ is a line module. 
  \item 
 If $\l \notin \{0,\infty\}$, then $(S/S\ell^\perp)[(KK')^{-1}]_0 \cong M(\l)$. 
\end{enumerate}
\end{proposition}
\begin{proof}
A simple calculation proves (1), and then (2) follows from \Cref{thm.pencil.lines}.

(3) The functor $j^*\pi^*:\Gr(S) \to \Mod(U_q(\fsl_2)$ defined by
$j^*\pi^*M = M[(KK')^{-1}]_0$ is exact, so
$(S/S\ell^\perp)[(KK')^{-1}]_0$ is isomorphic to
$S(KK')^{-1}]_0/(S\ell^\perp)[(KK')^{-1}]_0$.  Using the isomorphism
given by (\ref{eq.Uqsl2.iso}), it is clear that
$(S\ell^\perp)[(KK')^{-1}]_0$ is the left ideal of $U_q(\fsl_2)$
generated by $e$ and $k -\l$.
\end{proof}

\subsubsection{`Heretical' Verma modules} \Cref{pr.vrm-ln} illustrates
the importance of line modules for Artin-Schelter regular algebras
with Hilbert series $(1-t)^{-4}$. Line modules are just like Verma
modules. Indeed, Verma modules for $U(\fsl_2)$ and $U_q(\fsl_2)$ are
``affine pieces'' of line modules.

From the point of view of non-commutative projective algebraic
geometry, the line modules that correspond to Verma modules are no
more special than other line modules. One is tempted to declare that
if $\ell$ is any line on any $Q(\l)$, $\l\ne 0,\infty$, then
$(S/S\ell^\perp)[(KK')^{-1}]_0$ should be considered as a Verma
module.

Doing that would place $U_q(\fsl_2)$ on a more equal footing with
$U(\fsl_2)$: if one varies both the Borel subalgebra and the highest
weight, then $U(\fsl_2)$ has a 2-parameter family of Verma modules; if were to define Verma modules for $U_q(\fsl_2)$ as ``affine pieces'' of line modules, 
then $U_q(\fsl_2)$ would also have a 2-parameter family of Verma modules.

\subsubsection{Central (Casimir) elements} 
We define $\Omega(0) = \Omega(\infty) =KK'$ and, for each
$\l \in \CC-\{0,\infty\}$, we define
\begin{align*}
\Omega(\l)   \; :=\; &
EF \; + \; \frac{q^{-1}K^2+qK^{'2}}{(q-q^{-1})^2}  \;-\; \frac{q\l+q^{-1}\l^{-1}}{(q-q^{-1})^2}\,KK'
\\
\; =\; &
EF \; + \; \kappa^2 \big( q^{-1}K-q\l K')(K-\l^{-1}K')  
\\
\; =\; &
FE \; +  \; \kappa^2 \big( qK-q^{-1}\l^{-1} K')(K-\l K').  
\end{align*}
The elements $\Omega(\l)$, $\l \in \PP^1$, belong to the center of $S$
and span a 2-dimensional subspace of $S_2$.

We take note that $\Omega(\l)=\Omega(q^{-2}\l^{-1})$ and
$\Omega(\mu)\ne \Omega(\l)$ if $\mu \notin \{\l,q^{-2}\l^{-1}\}$.

Under the isomorphism $S[(KK')^{-1}]_0 \cong U_q(\fsl_2)$ given in
\Cref{eq.Uqsl2.iso}, we have
\begin{equation*}
  \Omega(\l)(KK')^{-1} \; = \; C  \;-\; \frac{q\l+q^{-1}\l^{-1}}{(q-q^{-1})^2}\,,  
\end{equation*}
where $C$ is the Casimir element defined in (\ref{eq.Casimir}).

The reader will notice similarities between the pencil of central subspaces $\CC\Omega(\l) \subseteq S_2$ and the pencil of quadrics $Q(\l) \subseteq \PP(S_1^*)$.
For example, exactly one $\Omega(\l)$ is a product of two degree-1 elements, namely $\Omega(0)=\Omega(\infty) =KK'$, and exactly one $Q(\l)$ that is
a union of two planes, namely $Q(0)=Q(\infty)=\{KK'=0\}$. In a similar vein, we expect that $S/(\Omega(\l))$ is a prime ring if and only if $\l$ is not 0 or $\infty$.
A formal connection between the $\Omega(\l)$'s and the $Q(\l)$'s is established in \cref{prop.Omega.ann.line}.

\begin{lemma}
\label{lem.Omega.ann.line}
Let $\l \in \CC^\times$. The central element $\Omega(\l)$ annihilates
$M_\ell$ for all lines $\ell$ of the form
$$
E -\kappa s(K-\l K') \; \,  = \,  \; s F+ \kappa (K-\l^{-1}K') \; = \; 0, \qquad s \in \PP^1. 
$$ 
\end{lemma}
\begin{proof}
  Let $s$ be any point on $\PP^1$. Since $\Omega(\l )$ equals
\begin{align*}
  FE \; +  \; & \frac{1}{(q-q^{-1})^2}\big( qK-q^{-1}\l ^{-1} K')(K-\l K')  
  \\
              & \; =\;  F\big(E-\kappa s(K-\l ^{-1}K')\big) \; + \; \kappa \big(qK-q^{-1}\l ^{-1}K'\big)\big(sF+\kappa (K-\l K')\big),
\end{align*}
it belongs to the left ideal generated by $E-\kappa s(K-\l ^{-1}K')$
and $sF+\kappa (K-\l K')$. That left ideal is $S\ell^\perp$ so, since
$\Omega(\l )$ is in the center of $S$, it annihilates $S/S\ell^\perp$.
\end{proof}

\begin{proposition}
\label{prop.Omega.ann.line}
Let $M_\ell$ be a line module.  If $\l \in \CC^\times$, then
$\Omega(\l)$ annihilates $M_\ell$ if and only if either
\begin{enumerate}
\item $\ell \subseteq Q(\l)$ and is in the same ruling as the line
  $E=K-\l K'=0$, or
\item $\ell \subseteq Q(q^{-2}\l^{-1})$ and is in the same ruling as
  the line $E=K-q^{-2}\l^{-1} K'=0$.
\end{enumerate}
Furthermore, $\Omega(0)=\Omega(\infty)$ annihilates $M_\ell$ if and
only if $\ell \subseteq Q(0)=Q(\infty)=\{KK'=0\}$.
\end{proposition}
\begin{proof}
  It is easy to see that the last sentence in the statement of the
  proposition is true so we will assume that
  $\ell \not\subseteq \{KK'=0\}$.  Since $M_\ell$ is a line module,
  $\ell$ lies on $Q(\mu)=Q(\mu^{-1})$ for some $\mu \in \CC^\times$.
  We fix such a $\l$.  Since $\ell \not\subseteq \{KK'=0\}$,
  $\mu \ne 0,\infty$.

($\Rightarrow$) Fix $\l \in \CC^\times$ and suppose that $\Omega(\l)$
annihilates $M_\ell$.

The lines on $Q(\mu)$ are given by (\ref{ruling.1}) and
(\ref{ruling.2}).  Replacing $\mu$ by $\mu^{-1}$ if necessary, we can
assume that $\ell$ belongs to the same ruling on $Q(\mu)$ as
$E=K-\mu K'=0$.  Hence $M_\ell$ is annihilated by $\Omega(\mu)$. If
$\Omega(\mu) \ne \Omega(\l)$, then $M_\ell$ is annihilated by
$KK'$. That is not the case, so $\Omega(\mu) = \Omega(\l)$. Hence
$\mu \in \{\l,q^{-2}\l^{-1}\}$. Hence either (1) or (2) holds.

($\Leftarrow$) This implication follows from
\Cref{lem.Omega.ann.line}.  If $\ell \subseteq Q(\l)$ and is in the
same ruling as the line $E=K-\l K'=0$, then $M_\ell$ is annihilated by
$\Omega(\l)$. If $\ell \subseteq Q(q^{-2}\l^{-1})$ and is in the same
ruling as the line $E=K-q^{-2}\l^{-1} K'=0$, then $M_\ell$ is
annihilated by $\Omega(q^{-2}\l^{-1})=\Omega(\l)$.
\end{proof}

We only care about the ideal generated by $\Omega(\l)$ and the matter
of which modules are annihilated by which $\Omega(\l)$'s. Thus, we
only care about $\Omega(\l)$ up to non-zero scalar multiples. It is
often better to think of $\Omega(\l)$ as an element in $\PP^1$.

\subsection{Fat points in $\Projnc(S)$ $\longleftrightarrow$ Finite dimensional simple  $U_q(\fsl_2)$-modules}
\label{ssect.fat.simples}
As the title suggests, this subsection establishes a connection
between the finite-dimensional simple $U_q(\fsl_2)$-modules $L(n,\pm)$
discussed in \Cref{ssse.fin} and certain fat points $F(n,\pm)$ of the
non-commutative scheme $\Projnc(S)$ that are defined below.
\cref{prop.F.simple} makes this connection explicit. We have not
addressed the question of whether the $F(n,\pm)$'s are all the fat
points.

\subsubsection{Some finite dimensional simple $S$-modules}
\label{ssect.fdiml}

We fix a square root, $\sqrt{q}$, of $q$ and adopt the convention that $q^{n/2-i}=(\sqrt{q})^{n-2i}$ and $q^{i-n/2}=(\sqrt{q})^{2i-n}$.
Let $V(n,\pm)$ be the vector space with basis $v_0,\ldots,v_n$ and define 
$$
Kv_i=\sqrt{\pm 1} \, q^{n/2-i} v_i, \qquad
K'v_i=\pm \sqrt{\pm 1} \, q^{i-n/2} v_i, 
$$
$$
Fv_i= \begin{cases} [n-i] v_{i+1} & \text{if $i<n$,} \\ 0 & \text{if $i=n$,} \end{cases} \qquad 
Ev_i= \begin{cases}\pm [i] v_{i-1} & \text{if $i>0$,} \\ 0 & \text{if $i=0$.} \end{cases}
 $$
 
 \subsubsection{Automorphisms of $S$ and auto-equivalences of
   $\Gr(S)$}
 \label{ssect.automs}
 Let $\theta :S \to S$ be the algebra automorphism defined by
 $\theta(K)=-K$, $\theta(K')=K'$, $\theta(E)=E$, $\theta(F)=F$.
 
 If $\ve \in \CC^\times$ let $\phi_\ve:S \to S$ be the algebra
 automorphism $\phi_\ve(a)=\ve^na$ for all $a \in S_n$.
 
 Let $\phi$ be a degree-preserving algebra automorphism of $S$. The
 functor $\phi^*:\Gr(S) \to \Gr(S)$ is defined as follows: if
 $M \in \Gr(S)$, then $\phi^*(M)$ is $M$ as a graded vector space and
 if $a \in S$ and $m \in M^*$, then $a\cdot m=\phi(a)m$. The functor
 $\phi^*$ is an auto-equivalence.

 \begin{proposition}
 \label{prop.Vnpm}
Let $\ve=-\sqrt{-1}$. 
 \begin{enumerate}
  \item 
 $V(n,\pm)$ is a simple $S$-module of dimension $n+1$.
  \item 
$V(n,\pm)$ is a $S[(KK')^{-1}]$-module with $(KK')^{-1}$ acting as the identity.  
  \item 
Identifying $U_q(\fsl_2)$ with $S[(KK')^{-1}]_0$ as in \Cref{prop.Uqsl2}, $V(n,\pm)\cong L(n,\pm)$ as a $U_q(\fsl_2)$-module.
\item
$\Omega(\pm q^n)$ annihilates $V(n,\pm)$.
\item 
$V(n,-) \cong \phi_{\ve}^*\,\theta^* V(n,+)$. 
\end{enumerate}
\end{proposition}
\begin{pf} 
(1)
First we check that the action makes $V(n,\pm)$ a left $S$-module.

If $v_{-1}=0$, then $EKv_i=\pm \sqrt{\pm 1} \, q^{n/2-i}[i]v_{i-1}$ and $KEv_i=\pm \sqrt{\pm 1} \, q^{n/2-i+1}[i]v_{i-1}=qEKv_i$.   
Hence $KE-qEK$ acts on $V(n,\pm)$ as 0.
With the understanding that $v_{n+1}=0$, ,  $FKv_i=\sqrt{\pm 1} \, q^{n/2-i}[n-i]v_{i-1}$ and $KFv_i=\sqrt{\pm 1} \, q^{n/2-i-1}[n-i]v_{i+1}=q^{-1}FKv_i$, so $KF-q^{-1}FK$ acts on $V(n,\pm)$ as 0.
Similar calculations show that $K'E-q^{-1}EK'$  and $K'F-qFK'$ act on $V(n,\pm)$ as 0 also.
Furthermore,
\begin{align*}
[E,F]v_i & \; =\; \pm \big(  [n-i] Ev_{i+1}-[i] Fv_{i-1} \big)
\\
& \;=\;   \pm \big(    [n-i] [i+1] [i] [n-i+1\big) v_{i} 
\\
& \;=\;  \pm [n-2i] v_{i} 
\\
&\; = \; \frac{K^2-K'^2}{q-q^{-1}} \, v_i ,
\end{align*}
so $V(n,\pm)$ really is a left $S$-module.

To see it is simple, first observe that the $v_i$'s are eigenvectors for $K$ with pairwise distinct eigenvalues. It follows that if $V(n,\pm)$ is not simple, then there
it has a proper submodule that contains some $v_i$. However, looking at the actions of $E$ and $F$ on the $v_j$'s, a submodule that contains one $v_i$ contains
all $v_i$'s. Hence $V(n,\pm)$ is simple.

(2)
Since $KK'$ acts on $V(n,\pm)$ as multiplication by $1$, the module-action of $S$ on $V(n,\pm)$ extends to a module-action of $S[(KK')^{-1}]$. 

(3)
Since $e=\frac{1}{\sqrt{q}}EK^{-1}$, $f=\frac{1}{\sqrt{q}}F(K')^{-1}$, and $k=K(K')^{-1}$, 
$$
kv_i \; = \; -q^{n-2i}v_i, 
$$
$$
fv_i= \begin{cases} \frac{1}{\sqrt{\pm q}}\, q^{n/2-i} [n-i] v_{i+1} & \text{if $i<n$,} \\ 0 & \text{if $i=n$,} \end{cases} \qquad 
ev_i= \begin{cases}\frac{\pm 1}{\sqrt{\pm q}}\, q^{i-n/2}  [i] v_{i-1} & \text{if $i>0$,} \\ 0 & \text{if $i=0$.} \end{cases}
$$

Choose non-zero scalars $\l_0,\ldots,\l_n$ such that 
$$
\l_{i-1}/\l_i \; = \; \sqrt{\pm q}\, q^{n/2-i}[n+1-i].
$$
The linear isomorphism $\phi:V(n,\pm) \to L(n,\pm)$ defined by $\phi(v_i)=\l_i m_i$ is a $U_q(\fsl_2)$-module isomorphism because $\phi(kv_i)=k\phi(v_i)$, 
$$
\phi(ev_i)  \; =\;  \frac{\pm 1}{\sqrt{\pm q}} \, q^{i-n/2}  [i] \l_{i-1}m_{i-1}
  \;=\;  \pm  [i][n+1-i]\l_i m_{i-1}
  \;=\;  e \phi(v_i),
$$
and
$$
\phi(fv_i)  \; =\;  \frac{1}{\sqrt{\pm q}} \, q^{n/2-i} [n-i] \l_{i+1} v_{i+1}
 \;=\;   \l_im_{i+1} 
 \;=\;   f \phi(v_i).
$$
Hence $V(n,\pm)\cong L(n,\pm)$ as claimed.

(4)
By Schur's Lemma, $\Omega(\l)$ acts on $V(n,\pm)$ as multiplication by a scalar. Thus, if $\Omega(\l)$ annihilates $v_0$ it annihilates $V(n,\pm)$.
Since $Ev_0=0$, $\Omega(\l)v_0=\kappa^2(qK-q^{-1}\l^{-1}K')(K-\l K')v_0$. The result follows from the fact that $(K \mp  q^nK')v_0=0$.

(5)
Let $v_0,\ldots,v_n$ be the basis  for $V(n,+)$ in \S\ref{ssect.fdiml} and, to avoid confusion, write $v_i'$ for the basis element $v_i$ in $V(n,-)$. 
Thus, $Kv_i'=-\ve q^{n/2-i}v_i'$.

Define $\psi:\phi_{\ve}^*\,\theta^* V(n,+) \longrightarrow V(n,-)$ by $\psi(v_i):=(-1)^i \ve^i v_i'$. To show $\psi $ is an $S$-module isomorphism it suffices to
show it is an $S$-module homomorphism. To this end, consider $v_i$ as an element in $\phi_{\ve}^* \theta^* V(n,+)$. 
Because $\theta\phi_\ve(K)=-\ve K$, $Kv_i=-\ve q^{n/2-i}v_i$. Hence  
$$
\psi(Kv_i)\; = \; \psi(-\ve q^{n/2-i}v_i) \; = \;  -\ve q^{n/2-i} (-1)^i \ve^i v_i'  \; = \;  (-1)^i \ve^i Kv_i' = K\psi(v_i).
$$
Similarly, because $\theta\phi_\ve(K')=\ve K'$ and $K'v_i'=\ve q^{i-n/2}v_i'$, 
$$
\psi(K'v_i)\; = \; \psi(\ve q^{i-n/2}v_i) \; = \;  \ve q^{i-n/2} (-1)^i  \ve^i v_i'  \; = \;  (-1)^i \ve^i K'v_i' = K'\psi(v_i).
$$
We also have 
$$
\psi(Fv_i)\; = \; \psi(\ve [n-i]v_{i+1}) \; = \;  \ve [n-i] (-1)^{i+1} \ve^{i+1} v_{i+1}'  \; = \;   (-1)^i  \ve^i Fv_i' = F\psi(v_i)
$$
and 
$\psi(Ev_i)\; = \; \psi(\ve [i]v_{i-1}) \; = \;  \ve [i] (-1)^{i-1}\ve^{i-1} v_{i-1}'  \; = \;   (-1)^i  \ve^i Ev_i' = E\psi(v_i)$.
\end{pf}

\subsubsection{Fat points and fat point modules}

For each $n \in \NN$ we define
$$
F(n,\pm) \; :=\; V(n,\pm) \otimes \CC[z]
$$
and make this a graded left $S$-module according to the recipe in \Cref{lem.old.Sm}. It is a fat point module.
\Cref{prop.F.simple} makes the statement that the fat point (module) $F(n,\pm)$ corresponds to the finite dimensional
simple $U_q(\fsl_2)$-module $L(n,\pm)$ precise.

\begin{lemma}
\label{lem.old.Sm.2}
If $\theta$ is the automorphism in \S\ref{ssect.automs}, then $\theta^*F(n,\pm) \cong F(n,\mp)$. 
\end{lemma}
\begin{proof}
If $V$ is any left $S$-module and $\phi_\ve$ the automorphism in \S\ref{ssect.automs} associated to 
$\ve \in \Bbbk^\times$, then the map $\Phi:V \otimes \Bbbk[z] \to (\phi^*_\ve V) \otimes \Bbbk[z]$, $\Phi(v \otimes z^i)=v \otimes (\ve z)^i$, is an isomorphism
in $\Gr(S)$. Hence $F(n,-)=\phi_\ve^* \theta^* V(n,+) \otimes \Bbbk[z] \cong \theta^* V(n,+) \otimes \Bbbk[z]  \cong \theta^*(V(n,+) \otimes \Bbbk[z]) = \theta^*F(n,+)$.
\end{proof}

\begin{proposition}
\label{prop.F.simple}
If $\pi^*:\Gr(S) \to \QGr(S)$ and $j^*:\QGr(S) \to U_q(\fsl_2)$ are the functors in \Cref{ssect.localize}, then $j^*\pi^*F(n,\pm) \cong L(n,\pm)$; i.e.,  
there is an isomorphism of $U_q(\fsl_2)$-modules
$$
F(n,\pm)[(KK')^{-1}]_0 \cong L(n,\pm).
$$
\end{proposition}
\begin{proof}
The functor $j^*\pi^*$ sends $M \in \Gr(S)$ to $M[(KK')^{-1}]_0$ where the latter is made into a $U_q(\fsl_2)$-module via the isomorphism $U_q(\fsl_2) \to S[(KK')^{-1}]_0$ in \Cref{prop.Uqsl2}. 

Since $KK'$ acts on $V(n,\pm)$ as the identity, it acts on $F(n,\pm)=V(n,\pm) \otimes k[z]$ as multiplication by $z^2$. Hence, 
$F(n,\pm)[(KK')^{-1}]_0=V(n,\pm) \otimes k[z,z^{-2}]_0 =V(n,\pm) \otimes 1$.

Let $\widehat{S}=S[(KK')^{-1}]$. 
Applying the functor $\widehat{S} \otimes_S -$ to the surjective $S$-module homomorphism $F(n,\pm) \to V(n,\pm)$, $v \otimes z^i \mapsto v$, produces
a surjective homomorphism  
$$
\psi: F[(KK')^{-1}]\; = \; \widehat{S}\otimes_S F(n,\pm) \, \longrightarrow \, \widehat{S}\otimes_S V(n,\pm)
$$
of $\widehat{S}$-modules. 
Of course, $\psi$ is a homomorphism of $\widehat{S}_0$-modules. Every homogeneous component of $F(n,\pm)[(KK')^{-1}]$ is an 
$\widehat{S}_0$-submodule of  $F(n,\pm)[(KK')^{-1}]$ so $\psi$ restricts to a homomorphism $F(n,\pm)[(KK')^{-1}]_0 \to \widehat{S}\otimes_S V(n,\pm)$
 of $\widehat{S}_0$-modules. But $\widehat{S}\otimes_S V(n,\pm)$ is isomorphic to $L(n,\pm)$ as an $\widehat{S}_0$-module by  \Cref{prop.Vnpm}(3)
 and, by the previous paragraph, $\dim(F(n,\pm)[(KK')^{-1}]_0)=\dim(V(n,\pm)) = n+1= \dim(L(n,\pm))$ so the restriction of $\psi$  to $F(n,\pm)[(KK')^{-1}]_0$
 is an isomorphism  of  $\widehat{S}_0$-modules.
\end{proof}

\begin{proposition}
\label{prop.lines.simples}
Let $n \ge 0$. 
Let $\ell_\pm$ be any line on $Q(\pm q^n)$  that is in the same ruling as the line $E=K \mp q^n K'=0$.
\begin{enumerate}
  \item 
  There is a surjective  $S$-module homomorphism $M_{\ell_{\pm}} \to V(n,\pm)$. 
  \item 
 There is a homomorphism $M_{\ell_{\pm}} \to F(n,\pm)$ in $\Gr(S)$ that becomes an epimorphism in $\QGr(S)$. 
 \item
 In $\Projnc(S)$, the fat point $F(n,\pm)$ lies on the line $\ell_{\pm}$. 
\end{enumerate}
\end{proposition}
\begin{proof}
Let $s \in \PP^1$ be such that $\ell_{\pm}$ is the line $\kappa(K\mp q^n K') - s^{-1}E  =  \kappa(K\mp q^{-n} K') + sF  =  0$.
Thus, 
$$
M_{\ell_{\pm}} \; \cong \;  \frac{S}{S X_{\pm} + S Y_{\pm}}  
$$
where $X_{\pm} =  \kappa(K\mp q^{-n }K') - s^{-1} E$ and $Y_{\pm}  =  \kappa(K\mp q^{n} K') + sF$.  

(1)
Since $V(n,\pm)$ is a simple $S$-module it suffices to show there is a non-zero homomorphism $M_{\ell_{\pm}} \to V(n,\pm)$.
For this,  it suffices to show there is a non-zero element in $V(n,\pm)$ annihilated by both $X_{\pm}$ and $Y_{\pm}$. 

If $s=0$, and $v_{\pm}=v_0\in V(n,\pm)$, then
$X_{\pm} v_{\pm}=Ev_0 =0$ and
$Y_{\pm} v_{\pm}=(K\mp q^n K' )v_{\pm} = 0$.  If $s=\infty$ and
$v_{\pm}=v_n\in V(n,\pm)$, then
$X_{\pm} v_{\pm}=(K \mp q^{-n}K') v_n=0$ and $Y_{\pm} v_{\pm}=Fv_n=0$.
Thus, (1) is true if $s$ equals $0$ or $\infty$.

From now on, assume that $s \ne 0,\infty$. Let $\l_0,\ldots,\l_n \in \Bbbk^\times$ be such that 
$$
\l_{i+1}/\l_i \; = \; \pm \sqrt{\pm 1}\, \frac{[n-i]}{[i+1]} s q^{-n/2}
$$
for all $i$. If
$$
v_{\pm} \; = \; \sum_{i=0}^n  \l_i v_i \; \in \; V(n, \pm),
$$
then 
\begin{align*}
X_{\pm} v_{\pm} &= \sum_{i=0}^n  \left(\kappa\sqrt{\pm 1} \, (q^{n/2-i} \mp q^{-n }q^{i-n/2}) \lambda_i v_i \, - \, s^{-1} (\pm 1) [i] \l_i v_{i-1}  \right) \\
                               &= \sum_{i=0}^n  \left( - q^{-n/2}\sqrt{\pm 1}\,  [n-i]  \lambda_i \, \mp \, s^{-1} [i+1] \l_{i+1}  \right) v_i \\
                & = 0
\end{align*}
and 
\begin{align*}
Y_{\pm} v_{\pm} &= \sum_{i=0}^n  \left(\kappa\sqrt{\pm 1} \, (q^{n/2-i} \mp q^{n }q^{i-n/2}) \lambda_i v_i \, + \, s  [n-i] \l_i v_{i+1}  \right) \\
                               &= \sum_{i=0}^n  \left( - q^{n/2}\sqrt{\pm 1}\,  [i]  \lambda_i \, + \, s  [n-i+1] \l_{i-1} \right) v_{i}  \\
                & = 0.
\end{align*}

(2)
By \Cref{lem.old.Sm}, the existence of a non-zero homomorphism $M_{\ell_{\pm}} \to V(n,\pm)$ implies the existence of a non-zero homomorphism
$M_{\ell_{\pm}} \to F(n,\pm)$ in $\Gr(S)$. However, as an object in $\QGr(S)$, $F(n,\pm)$ is irreducible so (2) follows.

(3)
This is just terminology.
\end{proof}

If one of the lines $\ell_{\pm} = \{X_{\pm}=Y_{\pm}=0\}$ in \Cref{prop.lines.simples} meets $C$ at $(\xi_1,\xi_2,\xi_3,0)$, then it meets $C'$ at 
$(q^{-n}\xi_1,q^n\xi_2,0,\pm \xi_3)$. Combining this with Theorem \ref{thm.pts.S}(5) gives the following result.

\begin{corollary}
\label{cor.epimorphism}
Let $p = (\xi_1,\xi_2,\xi_3,0) \in C$ and define $p_{\pm} = (\xi_1,\xi_2,0,\pm \xi_3)$. Let $\ell_\pm$ be the secant line to $C \cup C'$ 
passing through $p$ and $\sigma_S^n(p_{\pm})$. There is a surjective homomorphism $M_{\ell_{\pm}} \to V(n,\pm)$ in $\Mod(S)$ and 
an epimorphism  $M_{\ell_{\pm}} \to F(n,\pm)$  in $\QGr(S)$.
\end{corollary}

The analogue of \Cref{eq.BGG} requires results from the next section, and can be found in \Cref{th.fat_res}.

\section{Relation to the non-degenerate Sklyanin algebras}
\label{sec.Sklyanin}
We remind the reader that $S(\alpha,\beta,\gamma)$ denotes one of the non-degenerate Skylanin algebras defined in (\ref{eq.S.nondegenerate}). 

In this section, we show that some of our results about $S$ can be
obtained as ``degenerations'' of results in \cite{SS92,CS15-2,SSJ93}
about $S(\alpha,\beta,\gamma)$.  We also complete the characterization
of those line modules that surject onto fat point modules that we
alluded to in the last section.

\subsection{The point scheme of a non-degenerate Sklyanin algebra \cite{SS92}}
The point scheme of $S(\alpha,\beta,\gamma)$ embedded in $\PP^3$ with coordinates $x_0,x_1,x_2,x_3$ is 

\begin{equation}
\label{eq:nondeg.pts}
E'  \; = \;  E \, \cup \, \{(1,0,0,0), (0,1,0,0), (0,0,1,0),(0,0,0,1)\},
\end{equation}
where $E$ is the elliptic curve defined by
\begin{equation}
\label{eq:nondeg.EC}
x_0^2 + x_1^2 + x_2^2 + x_3^2 \;=\; 0 \;=\; x_3^2 + \frac{1-\gamma}{1+ \alpha} x_1^2 + \frac{1+ \gamma}{1- \beta} x_2^2.
\end{equation}
Equivalently, $E$ is the intersection of any two of the following quadrics:
\begin{equation}
\label{eq:nondeg.4quadrics}
\begin{aligned}
& x_0^2 & &+&  & x_1^2 & &+& & x_2^2 & &+& & x_3^2 &=& \quad 0 \\
& x_0^2 & &-& \beta \gamma & x_1^2 & &-& \gamma & x_2^2 & &+& \beta & x_3^2 &=& \quad 0 \\ 
& x_0^2 & &+& \gamma & x_1^2 & &-& \alpha \gamma & x_2^2 & &-& \alpha & x_3^2 &=& \quad 0 \\ 
& x_0^2 & &-& \beta & x_1^2 & &+& \alpha & x_2^2 & &-& \alpha & x_3^2 &=& \quad 0.
\end{aligned}
\end{equation}

There is an automorphism $\s$ of $E'$ that fixes the four isolated points and on $E$ is given by the formula 
\begin{equation}
\label{eq:nondeg.aut}
\sigma: 
\begin{pmatrix}
x_0 \\ x_1 \\ x_2 \\ x_3
\end{pmatrix}
\mapsto
\begin{pmatrix}
-2 \alpha \beta \gamma &x_1 x_2 x_3 &-& x_0(-x_0^2 + \beta \gamma x_1^2 + \alpha \gamma x_2^2 + \alpha \beta x_3^2) \\
2 \alpha &x_0 x_2 x_3 &+& x_1(\phantom{-}x_0^2 -
 \beta \gamma x_1^2 + \alpha \gamma x_2^2 + \alpha \beta x_3^2) \\
 2 \beta &x_0 x_1 x_3 &+& x_2(\phantom{-}x_0^2 +
 \beta \gamma x_1^2 - \alpha \gamma x_2^2 + \alpha \beta x_3^2) \\
 2 \gamma &x_0 x_1 x_2 &+& x_3(\phantom{-}x_0^2 +
 \beta \gamma x_1^2 + \alpha \gamma x_2^2 - \alpha \beta x_3^2) \\
\end{pmatrix}.
\end{equation}

\subsection{Degenerate point scheme}
\label{ssect.degen.pt.scheme}
In the degenerate case, substituting $(\alpha,\beta,\gamma) = (0,b^2, -b^2)$ into equations (\ref{eq:nondeg.pts}) through (\ref{eq:nondeg.aut}) yields the following results.

We will compare the point scheme of $S = S(0,b^2,-b^2)$ to
\begin{equation}
\label{eq:deg.pts}
E'_\dg \; := \;  E_\dg \cup \{(1,0,0,0), (0,1,0,0), (0,0,1,0),(0,0,0,1)\},
\end{equation}
where the curve $E_\dg$ is defined by
\begin{equation}
\label{eq:deg.EC}
x_0^2 + x_1^2 + x_2^2 + x_3^2 \;=\; 0 \;=\; x_3^2 + (1+b^2) x_1^2 + x_2^2,
\end{equation}
or as the intersection of any two of the quadrics
\begin{equation}
\label{eq:deg.4quadrics}
\begin{aligned}
& x_0^2 & &+&  & x_1^2 & &+& & x_2^2 & &+& & x_3^2 &=& \quad 0 \\
& x_0^2 & &+& b^4 & x_1^2 & &+& b^2 & x_2^2 & &+& b^2 & x_3^2 &=& \quad 0 \\ 
   & x_0^2 & &-& b^2 & x_1^2  &   & & &   &   & & & &=& \quad 0 \\ 
  & x_0^2 & &-& b^2 & x_1^2 &   & & &   &   & & &  &=& \quad 0.
\end{aligned}
\end{equation}

The automorphism on $E'_\dg$ fixes the four isolated points  and is defined on $E_\dg$ by
\begin{equation}
\label{eq:deg.aut}
\sigma_\dg: 
\begin{pmatrix}
x_0 \\ x_1 \\ x_2 \\ x_3
\end{pmatrix}
\mapsto
\begin{pmatrix}
    & & x_0(x_0^2 + b^4 x_1^2) \\
    & & x_1(x_0^2 +  b^4 x_1^2) \\
 \phantom{-} 2 b^2 x_0 x_1 x_3 &+& x_2(x_0^2 -  b^4 x_1^2) \\
 - 2 b^2 x_0 x_1 x_2 &+& x_3(x_0^2 -  b^4 x_1^2)
\end{pmatrix}
\end{equation}

\subsection{Comparison with our results}
We now compare $(E'_\dg,\sigma_\dg)$ with $(\cP_S, \sigma_S)$ from
Theorem \ref{thm.pts.S}. Recall our definitions of $E,F,K,K'$ from
(\ref{eq.notation}):
\begin{align}\label{eq:coords}
E = \frac{i}{2}(1 - ib)(x_2+ix_3), \qquad
F = \frac{i}{2}(1 + ib)(x_2-ix_3), \qquad
K = x_0 + b x_1, \qquad
K' = x_0 - b x_1.
\end{align}

With respect to the homogeneous coordinates $E$, $F$, $K$, and $K'$, 
\begin{align*}
\cP_S = C \cup C' \cup L \cup \{(0,0,1,\pm 1)\},
\end{align*}
where $C,C'$ and $L$ are given by
\begin{align*}
C'&: \qquad EF + \kappa^2 K'^2 = K\phantom{'} = 0, \\
C \phantom{'} &: \qquad EF + \kappa^2 K^2\phantom{'} = K' = 0, \\
L\phantom{'} &: \qquad\phantom{EF + \kappa^2} K\phantom{'^2} = K' = 0.
\end{align*}
The conics $C$ and $C'$ lie on the planes $K' = 0$ and $K = 0$, resp., and the line $L$ is the intersection of those two planes.
With respect to the homogeneous coordinates $E$, $F$, $K$, and $K'$,  (\ref{eq:deg.pts}) becomes
\begin{equation*}
  E'_\dg = E_\dg \cup \{ (0,0,1,1), (0,0,1,-1), (q,1,0,0),(-q,1,0,0)\}.  
\end{equation*}
The isolated points $(1,0,0,0)$ and $(0,1,0,0)$ in (\ref{eq:deg.pts})
remain isolated after degeneration, but the points $(0,0,1,0)$ and
$(0,0,0,1)$ in (\ref{eq:deg.pts}), which are $(q, 1, 0, 0)$ and
$(-q,1,0,0)$ in the $E,F,K,K'$ coordinates, become points on the line
$L$ in $\cP_S$ after degeneration.

Next, we compare $E_\dg$ with $C \cup C' \cup L$. The equation (\ref{eq:deg.EC}) yields
\begin{align*}
x_0^2 - b^2 x_1^2 = (x_0 - bx_1)(x_0 + b x_1) = K K' = 0.
\end{align*}
Hence $E_\dg \subseteq \{K = 0\} \cup \{K' = 0\}$.

On the plane $K' = 0$,  $x_0 = b x_1$ so both sides of equation (\ref{eq:deg.EC}) for $E_\dg$ become
\begin{align*}
(1+b^2) x_1^2 + x_2^2 + x_3^2 = 0.
\end{align*}
On the other hand, 
$C'$ is given by
\begin{align*}
0 = EF + \kappa^2 K'^2 &= -\frac{1}{4} (1+b^2)(x_2^2 + x_3^2) + 4 \kappa^2 b^2 x_1^2 \\
                        &= -\frac{1}{4} (1+b^2)(x_2^2 + x_3^2) - \frac{(1+b^2)^2}{4} x_1^2 \\
                        &= -\frac{1}{4}(1+b^2) \big(x_2^2 + x_3^2 + (1+b^2)x_1^2 \big).
\end{align*}
Hence $E_\dg \cap \{K' = 0\} = C'$.  A similar calculation yields the analogous result for the plane $K = 0$. We thus conclude that
\begin{align*}
E_\dg = C \cup C'.
\end{align*}

Finally, we compare $\sigma_\dg$ and $\sigma_S$. On the plane $K=0$,
\begin{align*}
\sigma_\dg:
\begin{pmatrix}
x_0 \\ x_1 \\ x_2 \\ x_3
\end{pmatrix}
\mapsto
\begin{pmatrix}
\phantom{-2b^4 x_1^2 x_2 +}(b^2 + b^4) x_1^2 x_0 \\ 
\phantom{-2b^4 x_1^2 x_2 +}(b^2 + b^4) x_1^3 \phantom{x_0} \\
\phantom{-}2 b^4 x_1^2 x_3 + (b^2 - b^4) x_1^2 x_2  \\
-2b^4 x_1^2 x_2 + (b^2 - b^4) x_1^2 x_3
\end{pmatrix}
= 
\begin{pmatrix}
(1+b^2) x_0 \\
(1+b^2) x_1 \\
(1-b^2) x_2 + 2b^2 x_3 \\
(1-b^2) x_3 -  2b^2 x_2.
\end{pmatrix}
\end{align*}
Changing coordinates,
\begin{align*}
\sigma_\dg:
\begin{pmatrix}
x_2 + i x_3 \\ x_2 - i x_3 \\ x_0 + b x_1 \\ 0
\end{pmatrix}
\mapsto
\begin{pmatrix}
(1-ib)^2 (x_2 + i x_3) \\
(1+ib)^2 (x_2 -ix_3) \\
(1+b^2) (x_0 + b x_1) \\
0
\end{pmatrix}
=
\begin{pmatrix}
q\phantom{^{-1}} (x_2 + i x_3) \\
q^{-1} (x_2 -ix_3) \\
\phantom{q^{-1}} (x_0 + b x_1) \\
0
\end{pmatrix}.
\end{align*}
Therefore,  in the $E,F,K,K'$ coordinates, $\sigma_\dg(\xi_1,\xi_2,\xi_3, 0) = (q \xi_1, q^{-1} \xi_2, \xi_3, 0) = \sigma_S(\xi_1,\xi_2,\xi_3, 0)$.
Similar calculations on the plane $K' = 0$ and on the isolated points yield $\sigma_\dg = \sigma_S$.

\subsection{Degenerations of Heisenberg automorphisms}
\label{subse.heis}

Recall (e.g., from \cite[Prop. 2.6]{CS15-2}) that the Heisenberg
group of order $4^3$ acts on the Sklyanin algebra
$S(\alpha,\beta,\gamma)$ as follows. 

First, fix square roots $a$, $b$ and $c$ of $\a$, $\b$ and $\c$
respectively. We define automorphisms $\phi_i$ of $S(\a,\b,\c)$ via

\begin{table}[htp]
\begin{center}
\begin{tabular}{|c|c|c|c|c|c|c|c|}
\hline
 & $x_0$ &  $x_1$ &  $x_2$ &  $x_3$  
\\
\hline
$\phi_1 $  & $bc x_1$ &   $-i x_0$ &  $-ib x_3$ &  $-c x_2$ $\phantom{\Big)}$
\\
\hline
$\phi_2$  & $ac x_2$ &   $-a  x_3$ &  $-ix_0$ &  $ -ic x_1$ $\phantom{\Big)}$
\\
\hline
$\phi_3$  & $ab x_3$ &   $-ia x_2$ &  $-bx_1$ &  $ -i  x_0$ $\phantom{\Big)}$
\\
\hline
\end{tabular}
\end{center} 
\caption{Automorphisms of $S(\a,\b,\c)$.}
\label{autom}
\end{table}

Now fix $\nu_1,\nu_2,\nu_3 \in k^\times$ such that $a\nu_1^2=b\nu_2^2=c\nu_3^2=-iabc$, and denote $\ve_1= \nu_1^{-1}\phi_1$, $\ve_2= \nu_2^{-1}\phi_2$,  $\ve_3 = \nu_3^{-1}\phi_3$, and $\d= i$. The subgroup $\langle \ve_1,\ve_2,\ve_3,\d\rangle \subseteq \mathrm{Aut}(S)$
is isomorphic to the Heisenberg group of order $4^3$, defined by
generators and relations as 
$$
H_4 \; : = \; \langle \ve_1,\ve_2,\d \; | \; \ve_1^4=\ve_2^4=\d^4=1, \;  \d\ve_1=\ve_1\d, \; \ve_2\d=\d\ve_2, \; \varepsilon_1\varepsilon_2=\delta \varepsilon_2\varepsilon_1\rangle.
$$

The algebras we are considering are of the form $S(0,\b,-\b)$. We define $c=ib$. The map $\phi_1$ still extends to an algebra automorphism
but $\phi_2$ and $\phi_3$ degenerate to {\it endo}morphisms. In terms of $E$, $F$, $K$, and $K'$, the endomorphisms $\phi_i$ act as 
in \Cref{autom}.

\begin{table}[htp]
\begin{center}
\begin{tabular}{|c|c|c|c|c|c|c|c|}
\hline
 & $E$ &  $F$ &  $K$ &  $K'$  
\\
\hline
$\phi_1 $  & $bqF$ &   $-bq^{-1}E$ &  $-ib K'$ &  $ibK$ $\phantom{\Big)}$
\\
\hline
$\phi_2$  & $\frac 12(1-ib)K'$ &   $\frac 12(1+ib)K$ &  $0$ &  $0$ $\phantom{\Big)}$
\\
\hline
$\phi_3$  & $\frac i2(1-ib)K'$ &   $-\frac i2(1+ib)K$ &  $0$ &  $0$ $\phantom{\Big)}$
\\
\hline
\end{tabular}
\end{center} 
\caption{Endomorphisms of $S(0,\b,-\b)$.}
\label{autom}
\end{table}

Although $\phi_2$ and $\phi_3$ are not isomorphisms, there are associated endo-functors $\phi_2^*$ and $\phi_3^*$ of $\Gr(S)$.
The application of $\phi_2^*$ and $\phi_3^*$ to the point modules  $M_{(0,0,1,\pm 1)}\in \Gr(S)$ produces point modules. 
Indeed, 
\begin{equation*}
  \phi_2(S) = \phi_3(S) = \bC[K,K']\subseteq S,
\end{equation*}
and the two point modules referred to above are cyclic $\bC[K,K']$-modules. With this in hand, the next result describes how
$\phi_i$ act on the four $S(0,\b,-\b)$-points obtained by degeneration
from $S(\a,\b,\c)$. The proof is a direct application of the formulas
in Table 2 above.

\begin{proposition}
The endomorphisms $\phi_i$ of $S$ described above twist the four
special point modules of $S$ as follows. 
\begin{enumerate}
  \item $\phi_1^*$ interchanges $M_{(0,0,1,\pm 1)}$ and interchanges
    $M_{(\pm q,1,0,0)}$;
  \item $\phi_2^*M_{(0,0,1,\pm 1)}\cong M_{(\pm q,1,0,0)}$;
  \item $\phi_3^*M_{(0,0,1,\pm 1)}\cong M_{(\mp q,1,0,0)}$.
\end{enumerate}

\end{proposition}

\subsection{Degenerations of fat point-line incidences}
\label{subse.fat}

In this section we describe resolutions of fat points by line modules by degenerating the analogous statements in \cite{SSJ93} for the elliptic algebras $S(\a,\b,\c)$.

If $\ell\subset \bP^3$ is the line passing through $p,p'\in C\cup C'$, we will sometimes denote the line module $M_\ell$ by $M_{p,p'}$ for clarity. We will also use the following notation: if $p=(\xi_1,\xi_2,\xi_3,0)\in C$, then $p_{\pm}=(\xi_1,\xi_2,0,\pm\xi_3)\in C'$ is the point for which $M_{p,p_{\pm}}$ surjects onto $F(0,\pm)$. Similarly, in order to keep the notation symmetric, if $p\in C'$ then $p_{\pm}$ is the point on $C$ for which $M_{p,p_{\pm}}$ surjects onto $F(0,\pm)$.  

Finally, we denote by $\sigma=\sigma_S: (C\cup C')^2\to (C\cup C')^2$ the diagonal action of $\sigma=\sigma_S$ on $C\cup C'$, and by $\psi$ the automorphism
\begin{equation*}
  \psi:=(\mathrm{id},\sigma):(C\cup C')^2\to (C\cup C')^2. 
\end{equation*}
By a slight abuse of notation, we use the same symbols to refer to the induced automorphisms on the variety of lines through pairs of points on $C\cup C'$.

\begin{theorem}\label{th.fat_res}
Let $n$ be a non-negative integer, $\ell_\pm$ a line through $p,p_{\pm}\in C\cup C'$, and $\ell_{\pm n}$ the line $\psi^n(\ell_\pm)$. 
In $\QGr(S)$, there is an exact sequence
  \begin{equation}\label{eq:fat_res}
    0 \to M_{\sigma^{-(n+1)}(\ell_{\pm n})}(-n-1) \to M_{\ell_{\pm}} \to F(n,\pm) \to 0. 
  \end{equation}
\end{theorem}
\begin{proof}
We will prove this for $\ell_+$. To that end, let $\ell$ be the line through $p$ and $p_+$.

The relation $\{(p,p_{\pm})\}$ on $C\cup C'$ is the fiber over
$(0,b^2,-b^2)$ of a family of relations over the space of parameters
$(\a,\b,\c)$ for the Sklyanin algebras. Specifically, let us write
\begin{equation*}
  (x_0,x_1,x_2,x_3)\mapsto (-x_0,x_1,x_2,x_3) 
\end{equation*}
for the $-$ maps on the elliptic curves $E=E(\a,\b,\c)$ and let
\begin{equation*}
  (x_0,x_1,x_2,x_3)\mapsto (x_0,x_1,-x_2,-x_3)
\end{equation*}
be addition by the $2$-torsion point $\omega\in E$.

{\bf Claim:} $\{(p,p_+)\}$ is the limit of the graphs of the maps
  $p\mapsto \omega-p$.
{\bf Proof:} 
In terms of the $x_i$ coordinates, the map $p\mapsto\omega-p$ amounts  to changing the sign of $x_0$. 
On the other hand the discussion at the beginning of \Cref{subse.fat} shows that in $(E,F,K,K')$-coordinates 
the map $p\mapsto p_+$ simply interchanges $K$ and $K'$. Since $C\cup C'$ is the degeneration of the family
\Cref{eq:nondeg.EC} of elliptic curves, the truth of the claim follows from the coordinate change formulas \Cref{eq:coords}.

The claim implies that the resolutions
\begin{equation*}
  0\to M_{\sigma(p),\sigma(\omega-p)}(-1)\to M_{p,\omega-p}\to \bullet\to 0
\end{equation*}
of the point modules associated to $(x_0,x_1,x_2,x_3) = (1,0,0,0)$
(e.g. from \cite[Thm. 5.7]{LS93}) degenerate to \Cref{eq:fat_res} for
$n=0$ in the $+$ case.

Similarly, for larger $n$ we have, in the non-degenerate case,
resolutions
\begin{equation*}
  0\to M_{\sigma^{-(n+1)}(p),\sigma^{-1}(\omega-p)}(-n-1)\to M_{p,\sigma^n(\omega-p)}\to \bullet\to 0
\end{equation*}
of $1$-critical fat points of multiplicity $n+1$ as explained in
\cite[Prop. 4.4(b)]{SSJ93}. These degenerate to a resolution of the
form \Cref{eq:fat_res} of a certain fat $S$-point module of
multiplicity $n+1$ (denoted momentarily by the same symbol $\bullet$):
\begin{equation}\label{eq:fat_res_bis}
    0 \to M_{\sigma^{-(n+1)}(\ell_n)}(-n-1) \to M_{\ell_n} \to \bullet \to 0,
\end{equation}
where $\ell$ is the line through $p$ and $p_+$ and
$\ell_n=\psi^n(\ell)$; note that $\bullet$ is the same fat point (up
to isomorphism in $\QGr(S)$) for all choices of $p$.

Finally, to argue that $\bullet\cong F(n,+)$ in the present case,
simply specialize to the line $\ell$ for which \Cref{eq:fat_res_bis}
is the homogenized version of the standard BGG resolution
\Cref{eq.BGG} of the simple $U_q(\fsl_2)$-module $L(n,+)$.

There is a similar argument for $F(n,-)$, or one can use the
observation in \Cref{lem.old.Sm.2} that
$F(n,-) \cong \theta^* F(n,+)$.
\end{proof}

\begin{remark}
  Note incidentally that we can obtain a proof of \Cref{pr.res_CC'} in
  the same spirit as that of \Cref{th.fat_res}, by degenerating the
  exact sequences
\begin{equation*}
	0\to M_{\sigma p,\sigma^{-1}p'}(-1)\to M_{p,p'}\to M_p\to 0
\end{equation*}
from \cite[Thm. 5.5]{LS93} for the Sklyanin algebras $S(\a,\b,\c)$,
where $p,p'$ belong to the elliptic curve component $E=E(\a,\b,\c)$ of
the point scheme of $S(\a,\b,\c)$ and $\sigma$ is the translation
automorphism of $E$. The result then follows from the observation made
above that $E(\a,\b,\c)$, together with its translation automorphism,
degenerates to $C\cup C'$ equipped with our automorphism (also denoted
by $\sigma$ throughout) when $\alpha\to 0$.
\end{remark}

The next result completes the description of the fat-point-line incidences.

\begin{proposition}\label{pr.fat_res}
  For $n\ge 0$ the line modules $M_{\ell_n}$ from \Cref{th.fat_res}
  are the only ones having $F(n,\pm)$ as a quotient in $\QGr(S)$.
\end{proposition}
\begin{proof}
  The only central element $\Omega(\l)$ annihilating $F(n,\pm)$ is
  $\Omega(\pm q^n)$. In turn, according to \Cref{prop.Omega.ann.line},
  the only line modules annihilated by this central element are the
  lines $M_{\ell_n}$ in question and the lines
  $M_{\sigma^{-(n+1)}(\ell_n)}$ appearing as the left-most terms in
  (\ref{eq:fat_res}). In conclusion, it suffices to show that there
  are no surjections
\begin{equation}\label{eq:fake}
  M_{\sigma^{-(n+1)(p)},\sigma^{-1}(p_{\pm})}\to F(n,\pm)
\end{equation}
in $\QGr(S)$. 

Let us specialize to $F(n,+)$, to fix notation. Upon localizing to
$S[(KK')^{-1}]_0 \cong U=U_q(\fsl(2))$, \Cref{eq:fake} becomes a
surjection
\begin{equation}\label{eq:fake_bis}
  \frac U{UX+UY}\to L(n,+),
\end{equation}
where $X=\kappa(1- q^{n+2}k^{-1}) - s^{-1}q^{-1/2}e$ and
$Y=\kappa(k-q^{-(n+2)}) + s q^{-1/2}f$ for some $s\in \bP^1$. If $s=0$
or $\infty$ then the left hand side of \Cref{eq:fake_bis} is the
simple Verma module of highest weight $q^{-(n+2)}$ (respectively
lowest weight $q^{n+2}$), thus contradicting the existence of such a
surjection. On the other hand, if $s\in \bC^\times$, then we obtain
surjections \Cref{eq:fake_bis} for {\it all} $s\in \bC^\times$ by
applying the $\bG_m$-action on $U$ given by
\begin{equation*}
  k\mapsto k,\quad e\mapsto s^{-1}e,\quad f\mapsto sf\quad \text{for}\quad s\in \bC^\times.
\end{equation*}
By continuity in $s\in \bP^1$, we then get such surjections for
$s=0,\infty$ as well, and the previous argument applies.
\end{proof}

We end with the following remark on certain modules over
$U=U_q(\fsl_2)$. Note that in the proof of \Cref{pr.fat_res} we showed
that the modules \Cref{eq:fake_bis} of the form
$\frac U{UX_{\pm}+UY_{\pm}}$ do not surject onto the simples
$L(n,\pm)$ for
\begin{equation}\label{eq:xy}
 X_{\pm}=\kappa(1\mp q^{n+2}k^{-1}) - s^{-1}q^{-1/2}e,\
 Y_{\pm}=\kappa(k\mp q^{-(n+2)}) + s
 q^{-1/2}f,\quad s\in \bP^1.
\end{equation}
In fact, we can do somewhat more:

\begin{proposition}\label{pr.simple}
  For $X_\pm$ and $Y_{\pm}$ as in \Cref{eq:xy} the module $\frac
  U{UX_{\pm}+UY_{\pm}}$ is simple. 
\end{proposition}
\begin{proof}
  As in the proof of \Cref{pr.fat_res}, we focus on $X=X_+$ and
  $Y=Y_+$ to fix notation. 

Assume otherwise. Then, using the equivalence between the category of
modules over $U\cong S[(KK')^{-1}]_0$ and a full subcategory of
$\QGr(S)$, this assumption implies that the line module 
\begin{equation*}
  M=M_{\sigma^{-(n+1)(p)},\sigma^{-1}(p_{\pm})} 
\end{equation*}
from \Cref{eq:fake} has a non-obvious subobject in $\QGr(S)$. The
criticality of line modules then implies that such a subobject would
then correspond to a shifted line module, and hence there is an
epimorphism of $M$ onto a non-zero fat point. Localizing back to $U$
this gives a surjection of $\frac U{UX+UY}$ onto a non-zero
finite-dimensional $U$-module, which we can contradict as in the proof
of \Cref{pr.fat_res}. 
\end{proof}

The significance of the proposition is that it fits the simple Verma
modules of highest and lowest weights $q^{-(n+2)}$ and respectively
$q^{n+2}$ into  ``continuous'' $\bP^1$-families of simple modules.


\begin{thebibliography}{10}

\bibitem{ATV1}
M.~Artin, J.~Tate, and M.~Van~den Bergh.
\newblock Some algebras associated to automorphisms of elliptic curves.
\newblock In {\em The {G}rothendieck {F}estschrift, {V}ol.\ {I}}, volume~86 of
  {\em Progr. Math.}, pages 33--85. Birkh\"auser Boston, Boston, MA, 1990.

\bibitem{ATV2}
M.~Artin, J.~Tate, and M.~Van~den Bergh.
\newblock Modules over regular algebras of dimension {$3$}.
\newblock {\em Invent. Math.}, 106(2):335--388, 1991.

\bibitem{BG02}
K.~A. Brown and K.~R. Goodearl.
\newblock {\em Lectures on algebraic quantum groups}.
\newblock Advanced Courses in Mathematics. CRM Barcelona. Birkh\"auser Verlag,
  Basel, 2002.

\bibitem{Chandler}
R.G. Chandler.
\newblock {\em On the Quantum Spaces of Some Quadratic Regular Algebras of
  Global Dimension Four}.
\newblock PhD thesis, University of Texas (Arlington), 2016.

\bibitem{CS15-2}
A.~{Chirvasitu} and S.~P. {Smith}.
\newblock {Exotic Elliptic Algebras of dimension 4 (with an Appendix by Derek
  Tomlin)}.
\newblock {\em ArXiv e-prints}, February 2015.

\bibitem{J96}
J.~C. Jantzen.
\newblock {\em Lectures on quantum groups}, volume~6 of {\em Graduate Studies
  in Mathematics}.
\newblock American Mathematical Society, Providence, RI, 1996.

\bibitem{Jimbo85}
M.~Jimbo.
\newblock A {$q$}-difference analogue of {$U({\mathfrak g})$} and the
  {Y}ang-{B}axter equation.
\newblock {\em Lett. Math. Phys.}, 10(1):63--69, 1985.

\bibitem{K95}
C.~Kassel.
\newblock {\em Quantum groups}, volume 155 of {\em Graduate Texts in
  Mathematics}.
\newblock Springer-Verlag, New York, 1995.

\bibitem{KS97}
A.~Klimyk and K.~Schm{\"u}dgen.
\newblock {\em Quantum groups and their representations}.
\newblock Texts and Monographs in Physics. Springer-Verlag, Berlin, 1997.

\bibitem{LBS93}
L.~Le~Bruyn and S.~P. Smith.
\newblock Homogenized {$\fsl(2)$}.
\newblock {\em Proc. Amer. Math. Soc.}, 118(3):725--730, 1993.

\bibitem{LBSvdB}
L.~Le~Bruyn, S.~P. Smith, and M.~Van~den Bergh.
\newblock Central extensions of three-dimensional {A}rtin-{S}chelter regular
  algebras.
\newblock {\em Math. Z.}, 222(2):171--212, 1996.

\bibitem{LS93}
T.~Levasseur and S.~P. Smith.
\newblock Modules over the {$4$}-dimensional {S}klyanin algebra.
\newblock {\em Bull. Soc. Math. France}, 121(1):35--90, 1993.

\bibitem{Lusz88}
G.~Lusztig.
\newblock Quantum deformations of certain simple modules over enveloping
  algebras.
\newblock {\em Adv. in Math.}, 70(2):237--249, 1988.

\bibitem{Lusz90}
G.~Lusztig.
\newblock On quantum groups.
\newblock {\em J. Algebra}, 131(2):466--475, 1990.

\bibitem{ShV02}
B.~Shelton and M.~Vancliff.
\newblock Schemes of line modules. {I}.
\newblock {\em J. London Math. Soc. (2)}, 65(3):575--590, 2002.

\bibitem{ShV02_bis}
B.~Shelton and M.~Vancliff.
\newblock Schemes of line modules. {II}.
\newblock {\em Comm. Algebra}, 30(5):2535--2552, 2002.

\bibitem{S94}
S.~P. Smith.
\newblock The four-dimensional {S}klyanin algebras.
\newblock In {\em Proceedings of {C}onference on {A}lgebraic {G}eometry and
  {R}ing {T}heory in honor of {M}ichael {A}rtin, {P}art {I} ({A}ntwerp, 1992)},
  number~1, pages 65--80, 1994.

\bibitem{SS92}
S.~P. Smith and J.~T. Stafford.
\newblock Regularity of the four-dimensional {S}klyanin algebra.
\newblock {\em Compositio Math.}, 83(3):259--289, 1992.

\bibitem{SSJ93}
S.~P. Smith and J.~M. Staniszkis.
\newblock Irreducible representations of the {$4$}-dimensional {S}klyanin
  algebra at points of infinite order.
\newblock {\em J. Algebra}, 160(1):57--86, 1993.

\bibitem{SVdB-NCQ}
S.~P. Smith and M.~Van~den Bergh.
\newblock Noncommutative quadric surfaces.
\newblock {\em J. Noncommut. Geom.}, 7(3):817--856, 2013.

\bibitem{VdB-blowup}
M.~Van~den Bergh.
\newblock Blowing up of non-commutative smooth surfaces.
\newblock {\em Mem. Amer. Math. Soc.}, 154(734):x+140, 2001.

\bibitem{Z-twist}
J.~J. Zhang.
\newblock Twisted graded algebras and equivalences of graded categories.
\newblock {\em Proc. London Math. Soc. (3)}, 72(2):281--311, 1996.

\end{thebibliography}
\bibliographystyle{plain}

\def\cprime{$'$} \def\cprime{$'$} \def\cprime{$'$} \def\cprime{$'$}
  \def\cprime{$'$}

\end{document}